\documentclass[12pt,a4paper]{amsart}
\usepackage[margin=1in]{geometry} 

\usepackage{amsmath,amssymb,amsthm,amscd}

\usepackage{titletoc}

\usepackage{tikz}
\usetikzlibrary{calc}
\usepackage[colorlinks=true, linkcolor=blue, anchorcolor=blue, citecolor=blue, filecolor=blue, menucolor= blue, urlcolor=blue,pdfencoding=auto, psdextra, pagebackref=true]{hyperref}
\usepackage{bookmark}

\usepackage{verbatim}
\usepackage{multirow}
\usepackage{mathtools}
\usepackage{enumitem}
\usepackage{pgf,tikz,mathrsfs}
\usetikzlibrary{arrows}
\usepackage[normalem]{ulem}
\usepackage{comment}
\usepackage{multicol}
\usepackage{xfrac,xcolor}
\usepackage{graphicx}
\numberwithin{equation}{section}
\usepackage{pstricks-add,pst-plot,pst-func}
\usepackage{chngcntr}
\usepackage{setspace}

\newtheorem{theorem}{Theorem}[section]
\newtheorem{proposition}[theorem]{Proposition}
\newtheorem{lemma}[theorem]{Lemma}
\newtheorem{corollary}[theorem]{Corollary}

\newtheorem{prob*}{Problem}

\newtheorem{conjecture}[theorem]{Conjecture}
\newtheorem*{theorem*}{Theorem}

\theoremstyle{definition}
\newtheorem{definition}[theorem]{Definition}

\newtheorem{example}[theorem]{Example}
\newtheorem{remark}[theorem]{Remark}

\usepackage{xparse}
\usetikzlibrary{matrix,backgrounds}
\pgfdeclarelayer{myback}
\pgfsetlayers{myback,background,main}

\newcommand{\field}{ \ensuremath{\mathbb{C}}}
\newcommand{\CC}{ \ensuremath{\mathbb{C}}}
\newcommand{\PP}{ \ensuremath{\mathbb{P}}}

\newcommand{\mO}{\mathcal{O}}

\DeclareMathOperator{\coker}{coker}

\DeclareMathOperator{\PGL}{PGL}
\DeclareMathOperator{\Sym}{Sym}

\DeclareMathOperator{\GGr}{\mathbb{G}r}
\DeclareMathOperator{\Pic}{Pic}

\newcommand{\rank}{\ensuremath{\mathrm{rank}}\hspace{1pt}}

\def\CC{{\mathbb C}}

\def\P#1{{\bf P}^#1}

\def\cocoa{{\hbox{\rm C\kern-.13em o\kern-.07em C\kern-.13em o\kern-.15em A}}}

\title{Weddle schemes}

\author[L.~Chiantini]{Luca Chiantini}
\address[L.~Chiantini]{Dipartimento di Ingegneria dell'Informazione e Scienze Matemati\-che, Universit\`a di Siena, Italy}
\email{luca.chiantini@unisi.it}

\author[{\L}.~Farnik]{{\L}ucja Farnik}
\address[{\L}.~Farnik]{Department of Mathematics, University of the National Education Commission, Krakow,
   Pod\-cho\-r\c a\.zych~2,
   PL-30-084 Krak\'ow, Poland}
\email{lucja.farnik@gmail.com}

\author[G.~Favacchio]{Giuseppe Favacchio}
\address[G.~Favacchio]{Dipartimento di Ingegneria, Universit\`a degli studi di Palermo,
Viale delle Scienze,  90128 Palermo, Italy}
\email{giuseppe.favacchio@unipa.it}

\author[B.~Harbourne]{Brian Harbourne}
\address[B.~Harbourne]{Department of Mathematics,
University of Nebraska,
Lincoln, NE 68588-0130 USA}
\email{brianharbourne@unl.edu}

\author[J.~Migliore]{Juan Migliore} 
\address[J.~Migliore]{Department of Mathematics,
University of Notre Dame,
Notre Dame, IN 46556 USA}
\email{migliore.1@nd.edu}

\author[T.~Szemberg]{Tomasz Szemberg}
\address[T.~Szemberg]{Department of Mathematics, University of the National Education Commission, Krakow,
   Podcho\-r\c a\.zych~2,
   PL-30-084 Kra\-k\'ow, Poland}
\email{tomasz.szemberg@gmail.com}

\author[J.~Szpond]{Justyna Szpond}
\address[J.~Szpond]{Department of Mathematics, University of the National Education Commission, Krakow,
   Podcho\-r\c a\.zych~2,
   PL-30-084 Krak\'ow, Poland}
\email{szpond@gmail.com}



\begin{document}

\begin{abstract}

The classical Weddle surface is the locus of vertices of quadric cones through six points in $\PP^3$ in linear general position. Equivalently, it is the closure of the locus of centers of projection from which those six points map to six points on a plane conic.
Motivated by this 1850 construction of T.\ Weddle, we introduce \emph{$d$-Weddle schemes} for finite point sets $Z\subset \PP^n$, defined by an analogous projection-to-degree-$d$ condition.
Our main tool is Macaulay duality, which yields a natural multiplication map in an Artinian algebra defined by powers of linear forms.
This viewpoint connects $d$-Weddle schemes to unexpected cones and interprets them as non-Lefschetz loci for these multiplication maps. Parallel to this, we give an analysis from the point of view of interpolation matrices, and we explain the connections between these approaches.

For a general set $Z\subset \PP^n$ of $\binom{d+n}{n}$ points, we show that the $d$-Weddle scheme is a hypersurface and we compute its degree. We also study general sets whose cardinalities are ``near" such a binomial coefficient, where the Weddle scheme has higher codimension. Returning to sets of six points (not always in linear general position), we discuss special configurations in which the appropriate Weddle scheme is reducible, or even nonreduced. 

\end{abstract}

\maketitle
\setcounter{tocdepth}{1}
\tableofcontents


\section*{Introduction}
 Among quartic surfaces, Weddle surfaces stand out as a distinguished and historically rich family, intertwining classical geometry with modern perspectives. Our aim is to provide a unified account of their classical theory and recent developments.

We begin with the mathematicians involved, as the development of mathematics is inseparable from the people who created it.

\subsection*{History} A central figure in nineteenth century projective geometry was Michel Chasles (1793--1880).
In 1837 he published \cite{Chasles1}, written in response to a prize question, posed by the Royal Academy of Brussels, on modern methods in geometry, with particular emphasis on reciprocal polars and related duality principles.
In Note~XXXIII (p.~403) Chasles asserted that the vertices of quadric cones through six given points in $\PP^3$ form a \emph{space curve} of degree~$3$:
\emph{``Le lieu g\'eom\'etrique des sommets des c\^ones du second degr\'e, qui passent tous par
six points donn\'es dans l'espace, est la courbe \`a double courbure, du troisi\`eme degr\'e,
d\'etermin\'ee par ces six points.''}

More than a decade later, in 1850, Thomas Weddle (1817--1853) published \cite{WEDDLE}.
He begins by observing that some results in an earlier paper of his had already been anticipated by Chasles, and he proceeds to obtain new ones.
He concludes with the following footnote, which corrects Chasles’s claim: the locus of vertices of quadric cones through six points is \emph{not} a cubic space curve, but a quartic surface.
\begin{quote}\footnotesize\itshape 
    \noindent\textasteriskcentered\ 
 I cannot forbear expressing a doubt, though with some diffidence, as to the correctness
of one of the theorems which M.\ Chasles has given in Note XXXIII. It is as follows:

The locus of the vertex of the cone of the second degree which passes through six
given points in space, is the curve of double curvature of the third degree determined by
(\emph{i.e.} passing through) these six points.

Now the analysis to which I have subjected the problem leads me to conclude that the
locus is a \emph{surface} of the fourth degree, which, as M.\ Chasles's own investigation
shews, must pass through the curve of double curvature of the third degree determined
by the six points. In fact, Chasles only proves the following theorem:\,\textemdash\,
Every cone which has its vertex on a curve of double curvature of the third degree, and
which passes through the curve, is of the second degree. And his error seems to have
arisen from momentarily imagining that every cone of the second degree passing through
the six points, will also pass through the curve of the third degree determined by these
points.
\end{quote}
\noindent Weddle did not give a proof of this statement, and one seems to have appeared only in 1861, when Arthur Cayley presented an analytic argument \cite{Cayley}. In 1870 Chasles again reported on the state of geometry in France under the auspices of the Minist\`ere de l'Instruction Publique, resulting in \cite{Chasles2}. He recalls the problem on p.~59,  but it is not until page p.~250 that he writes the fact that the locus is a surface of degree $4$. He says that this was stated by Weddle and proved by Cayley. 
On p.~92 he refers to Weddle as a \emph{``jeune g\'eom\`etre enlev\'e pr\'ematur\'ement aux sciences''} (``a young geometer taken too early from science'').
This quartic surface has come to be known as the \emph{Weddle surface}.
A detailed contemporary account of Weddle's life and work appears in the obituary notice published in \emph{Monthly Notices of the Royal Astronomical Society} (1854), \cite{RAS_Weddle}.

For the next sixty years or so the Weddle surface continued to attract attention, with at least fifteen papers devoted to its projective geometry, its plane sections, and its connections with other classical quartic surfaces and special point configurations. 
The Weddle surface also reappears in the modern literature  for several properties including being  birational to the Kummer surface \cite{CasselsFlynn} (this was classically known \cite{Baker1904, Bateman1905, Schottky}),  a quartic surface in $\mathbb{P}^3$ with the maximal possible number of sixteen ordinary double points, being the moduli space of some extension rank 2 vector bundles with trivial determinant on smooth genus $2$ curves \cite{Bolognesi2007}. Clemens \cite{Clemens1983} constructs some $4$-dimensional abelian varieties as intermediate Jacobians of desingularized double covers of $\PP^3$ branched over a Weddle surface.

\subsection*{Introduction to the problem.} Among the classical contributions, we are substantially inspired by Emch’s work \cite{EMCH}, which not only establishes additional facts about the Weddle quartic surface, but also suggests new directions for the subject. Our purpose in this paper is to give a modern treatment of the Weddle surfaces  and to develop new extensions from the viewpoint of the locus of vertices of cones.

We work over the field of complex numbers $\mathbb C$.
Given a finite set of points $Z\subset \PP^n$ and an integer $d\ge 1$, a natural interpolation question asks for the existence of degree $d$ cones containing $Z$ with a given vertex $P\in \PP^n$.
Under suitable assumptions on $Z$ and $d$, for a general point $P$, such cones typically do not exist; nevertheless, the locus of {\emph exceptional} points $P$ that do occur as vertices of degree $d$ cones through $Z$ can be nonempty and carries interesting geometry.

The first and most classical instance is the case $n=3$ and $d=2$ recalled above.
For six points $Z\subset \PP^3$ in linear general position, the vertices of quadric cones through $Z$ is a quartic surface, the \emph{Weddle surface}.
A convenient reformulation is in terms of projection:
a point $P\in \PP^3\setminus Z$ lies on the Weddle surface if and only if the projection
$\pi_P\colon \PP^3\dashrightarrow \PP^2$ maps $Z$ to six points lying on a plane conic.
This ``projection to a curve'' perspective was extended and revisited by Emch, Moore and others \cite{EMCH,moore}, and it suggests a general principle:
one can study loci of centers of projection from which a finite point set acquires an
unexpected degree-$d$ relation in the target.

Motivated by this idea, we introduce the \emph{$d$-Weddle locus} of $Z$ and its natural scheme structure,
the \emph{$d$-Weddle scheme} (see Section~\ref{sec:defWeddle}).

\medskip\noindent\textbf{Organization of the paper.}
In Section \ref{sec:2approaches} we describe two ways to construct the $d$-Weddle scheme of a given set of points:  via an interpolation matrix and via a matrix arising from Macaulay duality.
In Section~\ref{sec: general points in Pn} we start a systematic study of the Weddle schemes in $n$-dimensional projective spaces.   
In Section~\ref{six pt section} we discuss the geometry of the surfaces arising as Weddle schemes of six points in $\PP^3$ in more detail, showing that it changes even for points in linear general position. Dropping the LGP assumption leads to a menagerie of various surfaces of degree $4$, not necessarily reduced and irreducible.

In Section~\ref{sec:WeddleGeneralPoints} we compute $d$-Weddle loci for certain general point sets in $\PP^3$
of cardinality other than $6$, correcting a  mistake in \cite{EMCH}, and we then extend the discussion to
sets $Z\subset \PP^n$ for $n\ge 3$.

In Section~\ref{sec:LefschetzConnections} we explain the surprising equivalence with the non-Lefschetz schemes of \cite{BMMN}.
Finally, in Section~\ref{sec:Future} we discuss further directions and themes for future research.

\section{The \texorpdfstring{$d$}{d}-Weddle locus for a finite set of points in projective space}\label{sec:defWeddle}

By a cone of degree $d$ with vertex $P$ we mean a hypersurface $X\subset\PP^n$ of degree $d$ such that for every point $Q$ in $X$ the line joining $P$ and $Q$ is in $X$. This forces $P$ to have multiplicity $d$ in $X$.

Let $Z=\{P_1,\ldots,P_r\}\subset \PP^n$ be a set of distinct points and fix an integer $d\ge 1$.
Given a point $P\in \PP^n\setminus Z$, one can ask whether there exists a degree $d$ cone in $\PP^n$ with vertex $P$ that contains $Z$. For a general choice of $P$ the answer is often negative, but the set of \emph{exceptional} points
for which such cones exist (or exist in unexpectedly large families) can carry interesting geometry.
This section makes this precise and introduces the $d$-Weddle locus.

Let $R$ be the homogeneous coordinate ring of $\PP^n$, i.e., the polynomial ring
$\field[x_0,\ldots,x_n]$ over the ground field $\field$ with its standard grading by degree.
For a point $P\in \PP^n$ we write $I(P)\subset R$ for its homogeneous ideal, and similarly $I(Z)$ for the ideal of $Z$.

For $P\in \PP^n\setminus Z$ set
\[
I(Z,P,d) = I(Z)\cap I(P)^d \subset R.
\]
Thus a form $F\in [R]_t$ lies in $[I(Z,P,d)]_t$ if and only if $F$ vanishes on $Z$ and vanishes to order at least $d$ at $P$.

For $t\ge 0$ we set
\[
\delta(Z,P,d,t)\ =\ \dim_{\field}[I(Z,P,d)]_t.
\]
For fixed $Z,d,t$, the function $P\mapsto \delta(Z,P,d,t)$ is upper semicontinuous,
hence attains its minimum on a nonempty Zariski open subset of $\PP^n$.
We denote this generic value by
\[
\delta(Z,d,t)\ =\ \min_{P\in \PP^n\setminus Z}\delta(Z,P,d,t).
\]

The quantity $\delta(Z,P,d,d)$ has a direct geometric meaning: it counts the dimension of the linear system of degree $d$ hypersurfaces
having multiplicity $\ge d$ at $P$ and passing through $Z$, i.e., degree $d$ cones with vertex $P$ containing $Z$.

\begin{definition}
The {\it $d$-Weddle locus} of $Z$ is the Zariski closure  of the set of points $P \in \PP^n\setminus Z$
for which $\delta(Z,P,d,d)>\delta(Z,d,d)$.
\end{definition}

In other words,  the $d$-Weddle locus it is the locus of points $P$ for which there exist \emph{more} degree $d$ cones through $Z$ with vertex $P$ than for a general vertex.

\subsection*{Projection viewpoint}
The same definition can be phrased in terms of projections, which will be convenient throughout the paper.
Fix a general hyperplane $H\simeq \PP^{n-1}$ and, for $P\in \PP^n\setminus Z$, let
\[
\pi_P\colon \PP^n\dashrightarrow H
\]
be the linear projection from $P$. Let $Z_P:=\pi_P(Z)\subset H$ be the projected set.

A degree $d$ cone in $\PP^n$ with vertex $P$ containing $Z$ corresponds bijectively to a degree $d$ hypersurface in $H$ containing $Z_P$:
indeed, a hypersurface $D\subset H$ lifts to the cone over $D$ with vertex $P$, and conversely every cone with vertex $P$
is determined by its intersection with a general hyperplane $H$ not passing through $P$.
It follows that
\begin{equation}\label{eq:projection-equality}
\dim_{\field}[I(Z_P)]_d \ =\ \delta(Z,P,d,d),
\end{equation}
where $I(Z_P)\subset \field[H]$ is the homogeneous ideal of $Z_P$.

Thus the $d$-Weddle locus can equivalently be described as the closure of the set of points $P\notin Z$ such that
\[
\dim_{\field}[I(Z_P)]_d\ >\ \delta(Z,d,d).
\]
In this sense, the $d$-Weddle locus records those centers of projection for which the image $Z_P\subset H$
satisfies an \emph{unexpected} degree-$d$ relation.

In this sense, the $d$-Weddle locus measures the extent of this “unexpected” behavior in degree $d$.

This behaviour is closely analogous to the so-called non-Lefschetz  locus studied in \cite{BMMN}; we explain the precise relationship in Section~\ref{sec:LefschetzConnections}.

\subsection*{Linear general position and the classical case} Recall that a set of points $Z\subset\PP^n$ is in \textit{linear general position} (LGP) if every subset of $k\leq n+1$
points of $Z$ spans a $(k-1)$-linear subspace of $\PP^n$. 
\begin{example}[The classical Weddle surface] Let $Z\subset\PP^3$ be a set of 6 points  in LGP  and take $d=2$. 
For a general point $P\in \PP^3$ the projection $Z_P\subset \PP^2$ is a general set of $6$ points, hence
\[
\dim[I(Z_P)]_2=0\qquad\text{(equivalently, there is no conic through $Z_P$)}.
\]
However, there is a nonempty locus of points $P$ for which $Z_P$ \emph{does} lie on a conic; by \eqref{eq:projection-equality},
this is exactly the set of points $P$ for which $\delta(Z,P,2,2)=\dim[I(Z_P)]_2$ jumps from its generic value $0$ to~$1$.
The $2$-Weddle locus of $Z$ is the closure of this locus, that is the classical \emph{Weddle surface}, i.e., the quartic surface parametrizing vertices of quadric cones through $Z$. This surface is singular at the points of $Z$, see \cite[Introduction]{Edge}. It is depicted in Figure \ref{fig:weddle} prepared with the SURFER software \cite{surfer}, showing five of the six nodes (the six LGP points which define the surface are the only singularities and they are nodes; see Remark \ref{Remark4.9}).
In Remark~\ref{ClassicalWeddleResult} we give a new proof that this surface has degree~$4$.
\begin{figure}
    \centering
    \includegraphics[width=0.7\linewidth]{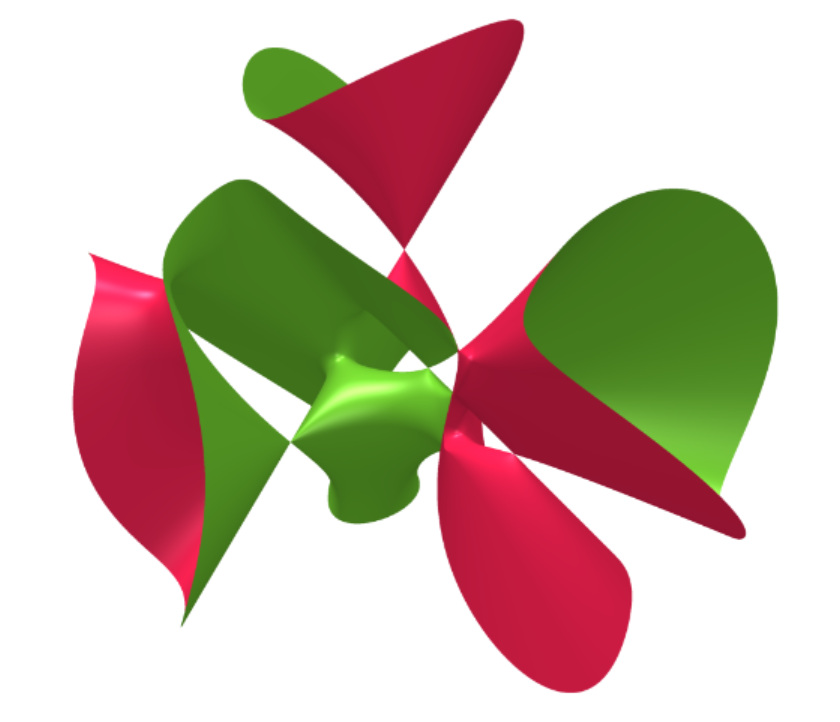}
    \caption{A Weddle surface of six points in $\mathbb P^3$ (5 of 6 nodes visible)} 
    \label{fig:weddle}
\end{figure}
\end{example}

In the next section we endow the $d$-Weddle locus with a natural scheme structure.

\section{The \texorpdfstring{$d$}{d}-Weddle scheme and two approaches to finding it}\label{sec:2approaches}

In this section we present two different approaches to compute the Weddle loci, and we show that they induce the same natural scheme structure, which we call the Weddle scheme.

One approach uses an interpolation matrix, where the relevant object is its kernel. This approach was used by Cayley \cite{Cayley} in the first published proof about the Weddle surface.  
The other uses a matrix arising from Macaulay duality, where the focus is in its cokernel.

The interpolation matrix is conceptually simpler but typically larger (hence less computationally efficient), whereas the Macaulay-duality matrix is often smaller and sometimes easier to use in practice, though less straightforward to write down explicitly.
Having both viewpoints is useful (see, e.g., 
Example \ref{BigWeddleExample} and Remark \ref{compare matrices}).

The following notation will be helpful.
Given a monomial $M=x_0^{i_0}\cdots x_n^{i_n}$ of degree $\deg(M)=i_0+\cdots+i_n$, let
\begin{equation}\label{notation1}
\begin{array}{lcl}
e_M &=&\displaystyle i_0!\cdots i_n!\\[4pt]
c_M &=&\displaystyle\frac{(i_0+\cdots+i_n)!}{i_0!\cdots i_n!}=\frac{\deg(M)!}{e_M}.
\end{array}
\end{equation}

\subsection{Interpolation matrices and $d$-Weddle schemes} \label{sec:Interpolation matrices}
Let $Z=\{P_1,\ldots,P_r\}\subset\PP^n$ be a finite set of distinct points.
Let $s_1,\ldots,s_r$ be positive integers.
Then $I=I(P_1)^{s_1}\cap\cdots\cap I(P_r)^{s_r}$ is the graded ideal of all forms that vanish to order at least $s_i$ at each point $P_i$.
The ideal $I$ is saturated  and defines a subscheme
$X\subset\PP^n$ known as a {\it fat point subscheme},
denoted $X=s_1P_1+\cdots+s_rP_r$, so we can write $I=I(X)$.
Since our ground field has characteristic~0, $[I]_t$ is given for each degree $t$ by 
the kernel of a matrix $\Lambda(X,t)$, known as 
the {\it interpolation matrix} for $X$ in degree $t$. We now explain this in more detail.

For a single point $P_1$, the vector space $[I(P_1)]_t$ is the span of all forms
$F\in [R]_t$ which vanish at $P_1$; i.e., such that $F(P_1)=0$.
If we enumerate the monomials of degree $t$ as $M_1,\ldots,M_N$, 
so $N=\binom{n+t}{n}$, then $F=\sum a_jM_j$ for some coefficients $a_j\in \field$.
Thus $F\in[I(P_1)]_t$ if and only if the coefficient vector
$(a_1,\ldots,a_N)$ has inner product 0 with 
the vector $(M_1(P_1),\ldots,M_N(P_1))$
of monomials of degree $t$ evaluated at $P_1$.
(We make sense of the value $M_j(P_1)$ by fixing a representative 
$(P_{01},\ldots,P_{n1})$ of the point $P_1=[P_{01}:\ldots:P_{n1}]$.)
Thus the interpolation matrix $\Lambda(P_1,t)=(M_1(P_1),\ldots,M_N(P_1))$
has a single row and $N$ columns. (Since this vector cannot be zero,
the kernel in this case has dimension $N-1$.)

Now consider $X=2P_1$. Now a form $F$ is in $I(X)$ if and only if it 
is singular at $P_1$.
This means $F\in I(X)$ if and only if $(\nabla(F))(P_1)=0$;
i.e., the gradient of $F$ evaluated at $P_1$ is zero.
The gradient $\nabla(F)$ of $F$ is 
$$\nabla(F)=\left(\frac{\partial F}{\partial x_0},\ldots,\frac{\partial F}{\partial x_n}\right).$$
Thus $F=\sum a_jM_j$ is in $[I(X)]_t$ if and only if 
$(\nabla(F))(P_1)=\sum a_j(\nabla(M_j))(P_1)=0$. So with respect to 
the partials $\frac{\partial }{\partial x_i}$ and the monomials
$M_j$, the matrix
$\Lambda(2P_1,t)$ is the $(n+1)\times N$-matrix 
whose entries are 
$$\Lambda(2P_1,t)_{ij}=\frac{\partial M_j}{\partial x_i}(P_1).$$

More generally, given a monomial $m=x_0^{i_0}\cdots x_n^{i_n}$,
let $\partial_m$ denote the differential operator 
$$\partial_m=\frac{\partial^{i_0}}{\partial x_0^{i_0}}
\cdots\frac{\partial^{i_n}}{\partial x_n^{i_n}}.$$
(We can extend this to arbitrary forms by linearity:
given an enumeration $m_i$ of the monomials of degree $k$ 
and a form $F=\sum a_im_i$, we denote by $\partial_F$
the operator $\sum a_i\partial_{m_i}$.)

We now define $\nabla_k$ as
the vector whose components are the $\partial_{m_i}$.
Thus $\nabla_1=\nabla$ and we define $\nabla_0=1$.
Generalizing the case of multiplicity 1 and 2, for $X=(k+1)P_1$
we have that $F\in [I(X)]_t$ if and only if $(\nabla_k(F))(P_1)=0$, so 
$\Lambda((k+1)P_1,t)$ is the $\binom{n+k}{n}\times N$-matrix 
whose entries are 
$$\Lambda((k+1)P_1,t)_{ij}=\frac{\partial M_j}{\partial m_i}(P_1)=\partial_{m_i}M_j(P_1).
$$

Thus the interpolation matrix $\Lambda(X,t)$ for $X=s_1P_1+\cdots+s_rP_r$
is a block matrix of size $(\sum_i\binom{n+s_i-1}{n})\times N$
whose blocks are $\Lambda(s_iP_i,t)$, arranged vertically; i.e.,
$$\Lambda(X,t)=
\left(
\begin{array}{c}
\Lambda(s_1P_1,t)\\
\vdots\\
\Lambda(s_rP_r,t)\\
\end{array}
\right).$$

For a fixed $Z=P_1+\cdots+P_r$ and an additional fat point $dP$, 
$\Lambda(Z+dP,t)$ is an $\left(r+\binom{n+d-1}{n}\right)\times N$ matrix. 
For $t\geq d$, its entries are either
scalars, or monomials of degree $t-d+1$ in the coordinates of $P$.
If we regard the coordinate variables as being the coordinates of $Q$,
so $Q=[x_0:\ldots:x_n]$, then the entries of $\Lambda(Z+dQ,t)$
are monomials in the variables $x_i$, 
and we can specialize the matrix $\Lambda(Z+dQ,t)$ to $\Lambda(Z+dP,t)$
for a particular point $P=[p_0:\ldots:p_n]$ by plugging each $p_i$ 
into the variable $x_i$.
For fixed $Z$, $d$ and $t$, 
the rank of $\Lambda(Z+dP,t)$ achieves its maximum $\rho(Z,d,t)$ 
when $P$ is general. In particular, we have 
$$\rho(Z,d,t)=N-\delta(Z,d,t).$$

In order to introduce the $d$-Weddle scheme, we work with the interpolation matrix for $t=d$:
$$\Lambda(Z+dQ,d)=
\left(
\begin{array}{c}
\Lambda(P_1,d)\\
\vdots\\
\Lambda(P_r,d)\\
\Lambda(dQ,d)\\
\end{array}
\right)=
\left(
\begin{array}{c}
\Lambda(Z,d)\\
\Lambda(dQ,d)\\
\end{array}
\right),$$
where $\Lambda(Z,d)$ is the submatrix consisting of
the first $r$ rows of $\Lambda(Z+dQ,d)$ and thus
corresponds to the points $P_i$ of $Z=P_1+\cdots+P_r$, 
while $\Lambda(dQ,d)$ consists of the remaining $\binom{n+d-1}{n}$ rows and 
corresponds to $dQ$.

Fix an ordering $m_1,\ldots,m_{\binom{n+d-1}{n}}$ of the monomials of degree $d-1$ and an ordering $M_1,\ldots,M_{\binom{n+d}{n}}$ of the monomials of degree $d$. Then the rows of $\Lambda(dQ,d)$ are indexed by the $m_i$ and  the columns of $\Lambda(Z+dQ,d)$ are indexed by the  $M_j$.

For the rows corresponding to points $P_i\in Z$, the $(i,j)$-entry is simply the evaluation $M_j(P_i)$.

The entry for the row for $m_i$ and the column for $M_j$ is  
\[
\partial_{m_i}(M_j)(Q).
\]
This entry is $0$ unless $m_i$ divides $M_j$; if $m_i\mid M_j$, then $M_j=m_i\,x_{k}$ for a unique variable $x_k$, and in that case
\[
\partial_{m_i}(M_j)(Q)=e_{M_j}\,x_k,
\]
where, as in Notation \eqref{notation1},  $e_{M_j}=i_0!\cdots i_n!$ for $M_j=x_0^{i_0}\cdots x_n^{i_n}$. 
(See Example~\ref{BigWeddleExample} for an explicit construction of $\Lambda(Z+dQ,d)$).

The $d$-Weddle locus is the closure of the set of points $P\notin Z$ at which the matrix
$\Lambda(Z+dP,d)$ fails to have the generic rank $\rho(Z,d,d)$.
Therefore it is cut out by the ideal
\[
I_{\rho(Z,d,d)}\bigl(\Lambda(Z+dQ,d)\bigr),
\]
generated by the $\rho(Z,d,d)\times \rho(Z,d,d)$ minors of $\Lambda(Z+dQ,d)$.
We call this ideal the {\it $d$-Weddle ideal} for $Z$.

\begin{definition}
Let $Z\subset\PP^n$ be a finite set of points.
The {\it $d$-Weddle scheme} for $Z$ is the scheme defined by
the saturation of the $d$-Weddle ideal.
\end{definition}

\begin{remark}\label{r. ReducedIntMat}
Computing ideals of minors is more efficient for smaller matrices, so it can be useful to reduce the size of the interpolation matrix.
It is easy to check that elementary row and column operations do not change ideals of minors. Row operations applied to $\Lambda(Z,d)$
can be used to bring $\Lambda(Z,d)$ 
first into block form as follows
$$\Lambda'({Z,d})=
\left(
\begin{array}{c|c}
id_\alpha & *\\
\hline
0 & 0\\
\end{array}
\right)
$$
where $id_\alpha$ is the $\alpha\times\alpha$ identity 
matrix with
$\alpha$ being the rank of $\Lambda(Z,d)$
(so the number of zero rows above is $r-\alpha$),
and where the 0's represent 0 matrices of appropriate sizes
and $*$ represents a submatrix whose entries 
are not of interest (since they will eventually be zeroed out
with column operations). Row operations applied to
$$\left(
\begin{array}{c}
\Lambda'({Z,d})\\
\Lambda(dQ,d)\\
\end{array}
\right)$$
reduce it to the form
$$\left(
\begin{array}{c|c}
id_\alpha & *\\
\hline
0 & 0\\
\hline
0 & \Lambda'({Z+dQ,d})\\
\end{array}
\right)
$$
and now column operations give
$$\left(
\begin{array}{c|c}
id_\alpha & 0\\
\hline
0 & 0\\
\hline
0 & \Lambda'({Z+dQ,d})\\
\end{array}
\right),
$$
where $\Lambda'({Z+dQ,d})$ is a 
$\binom{n+d-1}{n}\times (\binom{n+d}{n}-\alpha)$
matrix of linear forms (as shown in Example \ref{BigWeddleExample}).
Since the ideal of $q\times q$ minors is the same for this last matrix and for $\Lambda(Z+dQ,d)$ for every $q$,
we see $I_q(\Lambda(Z+dQ,d))$ is the trivial ideal $(1)$ for $q\leq\alpha$.
If for $q=\alpha+1$ we have $I_q(\Lambda(Z+dQ,d))=(0)$,
then $\Lambda'({Z+dQ,d})$ is a zero matrix and the $d$-Weddle
locus is empty. Otherwise, for $q>\alpha$ we have
$$I_q(\Lambda(Z+dQ,d))=
I_{q-\alpha}(\Lambda'({Z+dQ,d}))+
I_{q-\alpha+1}(\Lambda'({Z+dQ,d}))+\cdots.$$
Thus for the largest $q$ such that $I_{q-\alpha}(\Lambda'({Z+dQ,d}))$
is nonzero we have
$I_q(\Lambda(Z+dQ,d))=I_{q-\alpha}(\Lambda'({Z+dQ,d}))$
so we see that this $q$ is $q=\rho(Z,d,d)$, hence
$$I_{\rho(Z,d,d)}(\Lambda(Z+dQ,d))=
I_{\rho(Z,d,d)-\alpha}(\Lambda'({Z+dQ,d}))$$
both define the $d$-Weddle scheme. 
\end{remark}

\subsection{Macaulay duality and $d$-Weddle schemes}\label{sec:Macaulay duality}
The $d$-Weddle scheme can also be obtained using 
Macaulay duality 
(also called inverse systems; see \cite{EI}),
as we now explain.

Consider the polynomial rings $R=\field[x_0,\ldots,x_n]$ and
$R^*=\field[\partial_{x_0},\ldots,\partial_{x_n}]$, where formally we think of
the differential operators $\partial_{x_i}$ as independent indeterminates. 
Thus given a point $P=[p_0:\ldots:p_n]\in \PP^n$,
we get the dual form $L_P=p_0x_0+\cdots+p_nx_n\in [R]_1$ 
and the element $\partial_{L_P}=\sum p_i\partial_{x_i}\in [R^*]_1$.

Macaulay duality comes from regarding $R^*$ as acting on $R$.
Given a point $P=[p_0:\ldots:p_n]\in \PP^n$ and $0\leq k\leq t$, the annihilator
of $[I(P)^k]_t$ under this action is $[(\partial_{L_P}^{t-k+1})]_t$, hence we have 
$$[I(P)^k]_t\cong [R^*/(\partial_{L_P}^{t-k+1})]_t,$$
where $\cong$ denotes $\field$-vector space isomorphism.
More generally, given points $P_i\in\PP^n$ and a point $P\in\PP^n$, and
integers $0\leq k_i\leq t$ and $0\leq d\leq t$, we have
$$[I(P_1)^{k_1}\cap\cdots\cap I(P_r)^{k_r}]_t\cong [R^*/(\partial_{L_{P_1}}^{t-k_1+1},\ldots,\partial_{L_{P_r}}^{t-k_r+1})]_t$$
and
$$[I(P_1)^{k_1}\cap\cdots\cap I(P_r)^{k_r}\cap I(P)^{d}]_t\cong [R^*/(\partial_{L_{P_1}}^{t-k_1+1},\ldots,\partial_{L_{P_r}}^{t-k_r+1},\partial_{L_P}^{t-d+1})]_t.$$

We have the following exact sequence induced by the multiplication map $\times \partial_{L_P}^{t-d+1}$
{\footnotesize
\begin{equation}\label{OrigMacDualSeq}
\left[\frac{R^*}{(\partial_{L_{P_1}}^t,\ldots,\partial_{L_{P_r}}^t)}\right]_{d-1} \xrightarrow{\times\partial_{L_P}^{t-d+1}} 
\left[\frac{R^*}{(\partial_{L_{P_1}}^t,\ldots,\partial_{L_{P_r}}^t)}\right]_t
\to 
\left[\frac{R^*}{(\partial_{L_{P_1}}^t,\ldots,\partial_{L_{P_r}}^t,\partial_{L_P}^{t-d+1})}\right]_t
\to 0
\end{equation}
}
(where $\times\partial_{L_P}^{t-d+1}$ denotes the map given by multiplication by $\partial_{L_P}^{t-d+1}$). Since $0 < d \leq t$, we have
$$[R]_{d-1}\cong [R^*]_{d-1}=[R^*/(\partial_{L_{P_1}}^t,\ldots,\partial_{L_{P_r}}^t)]_{d-1},$$
$$[R^*/(\partial_{L_{P_1}}^t,\ldots,\partial_{L_{P_r}}^t)]_t\cong [I(P_1)\cap\cdots\cap I(P_r)]_t$$ 
and
$$[R^*/(\partial_{L_{P_1}}^t,\ldots,\partial_{L_{P_r}}^t,\partial_{L_P}^{t-d+1})]_t\cong [I(P_1)\cap\cdots\cap I(P_r)\cap I(P)^d]_t.$$
In particular, as a vector space, $[I(P_1)\cap\cdots\cap I(P_r)\cap I(P)^d]_t$ is isomorphic to the cokernel of the map $\times\partial_{L_P}^{t-d+1}$.

For the $d$-Weddle locus we want $t=d$, in which case the sequence \eqref{OrigMacDualSeq} is
{\footnotesize
\begin{equation}\tag{\ref*{OrigMacDualSeq}$^\prime$}\label{OrigMacDualSeqPrime}
\left[\frac{R^*}{(\partial_{L_{P_1}}^d,\ldots,\partial_{L_{P_r}}^d)}\right]_{d-1} \xrightarrow{\times\partial_{L_P}} 
\left[\frac{R^*}{(\partial_{L_{P_1}}^d,\ldots,\partial_{L_{P_r}}^d)}\right]_d
\to 
\left[\frac{R^*}{(\partial_{L_{P_1}}^d,\ldots,\partial_{L_{P_r}}^d,\partial_{L_P})}\right]_d
\to 0.
\end{equation}
}
\normalsize
In preparation for comparing what this sequence gives to what the interpolation matrix gives, 
it is helpful to rewrite it as the exact sequence 
\begin{equation}\label{NewMacDualSeq}
([R^*]_0)^r \oplus[R^*]_{d-1}\xrightarrow{D \oplus (\times\partial_{L_P})} [R^*]_d\to   
[R^*/(\partial_{L_{P_1}}^d,\ldots,\partial_{L_{P_r}}^d,\partial_{L_P})]_d\to 0
\end{equation}
where
$$([R^*]_0)^r \xrightarrow{D} [R^*]_d$$
is the map $v=(a_1,\ldots,a_r)\in ([R^*]_0)^r\mapsto D(v)=a_1\partial_{L_{P_1}}^d+\cdots+a_r\partial_{L_{P_r}}^d$,
and
$$[R^*]_{d-1} \xrightarrow{\times\partial_{L_P}} [R^*]_d$$
is the multiplication map 
$w\in [R^*]_{d-1}\mapsto (\times\partial_{L_P})(w)=w\partial_{L_P}$,
hence
$$(D \oplus (\times\partial_{L_P}))(v\oplus w)=D(v)+(\times\partial_{L_P})(w).$$

So now $[I(P_1)\cap\cdots\cap I(P_r)\cap I(P)^d]_d$ is isomorphic to the vector space cokernel of the map $D \oplus (\times\partial_{L_P})$.
If we regard $[R^*]_{d-1}$ as being the sum $\oplus_m [R^*]_0$ over all monomials $m$ of degree $d-1$ and
$[R^*]_d$ as being the sum $\oplus_M [R^*]_0$ over all monomials $M$ of degree $d$, then 
$$([R^*]_0)^r\oplus [R^*]_{d-1} \xrightarrow{D \oplus (\times\partial_{L_P})} [R^*]_d$$
can (in terms of the bases of monomials $m$ and $M$) be written as a matrix map $T=T(Z+dP,d)$
\begin{equation}\label{MatrixT}
([R^*]_0)^r\bigoplus \oplus_m [R^*]_0 \xrightarrow{T=[T_1 | T_2]} \oplus_M [R^*]_0,
\end{equation}
where $T_1$ is the matrix for $D$ giving the map
$([R^*]_0)^r\xrightarrow{T_1} \oplus_M [R^*]_0$ and
$T_2$ is the matrix for $\times\partial_{L_P}$ giving the map
$[R^*]_{d-1}\xrightarrow{T_2} \oplus_M [R^*]_0$.

The entries of $T_1$ are scalars given in terms of the coordinates of the points
$P_i=[p_{0i}:\ldots:p_{ni}]$.
The entry of $T_1$ corresponding to the point $P_i$ and the monomial $M$ is
$(T_1)_{M,i}=c_MM(P_i)$, where
$c_M$ comes (via the Binomial Theorem) from the expansion $\partial_{L_{P_i}}^d=(p_{0i}\partial_{x_0}+\cdots+p_{ni}\partial_{x_n})^d=\sum_Mc_MM(P_i)\partial_M$.
Recall that $c_M$ is the binomial coefficient corresponding  to the monomial $M$
(hence for $M=x_0^{i_0}\cdots x_n^{i_n}$, we have $c_M=\frac{(i_0+\cdots+i_n)!}{i_0!\cdots i_n!}=d!/e_M$).
Finally, $M(P_i)$ is the value of $M$ at $P_i$ (so plug $p_{ji}$ into each variable $x_j$
appearing in $M$).

The entry of $T_2$ corresponding to the monomials $m$ and $M$ is
$(T_2)_{M,m}=0$ unless $mx_i=M$, and then $(T_2)_{M,m}=p_i$, 
since $\times\partial_{L_P}\colon \partial_m\mapsto \partial_{mL_P}=\sum_i p_i\partial_{mx_i}$.

Note that we have 
\begin{align*}
\binom{d+n}{n}-\rank\Lambda(Z+dP,d)=&\dim \ker \Lambda(Z+dP,d)\\
=\dim\coker T(Z+dP,d) =&\binom{d+n}{n}-\rank T(Z+dP,d)
\end{align*}
since both the kernel and cokernel are isomorphic to
$[I(P_1)\cap\cdots\cap I(P_r)\cap I(P)^d]_d$,
and hence the interpolation matrix $\Lambda(Z+dP,d)$ has the same rank as
the Macaulay duality matrix $T(Z+dP,d)$.

The rank of $T(Z+dP,d)$ achieves its maximum $\tau(Z,d)$ 
when $P$ is general; indeed $\tau(Z,d)=\rho(Z,d,d)$.
Thus the $d$-Weddle locus is the closure of the locus of points $P\not\in Z$
such that $\rank(T(Z+dP,d))<\tau(Z,d)$.
This locus is defined by the ideal $I_{\tau(Z,d)}(T(Z+dQ,d))$
of $\tau(Z,d)\times\tau(Z,d)$ minors of $T(Z+dQ,d)$
for $Q=[x_0:\ldots:x_n]$.
Note that the largest $s$ such that $I_s(T(Z+dQ,d))\neq(0)$
is $s=\tau(Z,d)$, and the $d$-Weddle locus for $Z$ is
the zero locus of $I_s(T(Z+dQ,d))$ for this $s$.

\subsection{Comparing the two approaches} What is true, but not obvious, is that $I_s\bigl(T(Z+dQ,d)\bigr)$ coincides with the $d$-Weddle ideal defined earlier; that is, it is not a priori clear that
$$
I_{\tau(Z,d)}\bigl(T(Z+dQ,d)\bigr)
=
I_{\rho(Z,d)}\bigl(\Lambda(Z+dQ,d)\bigr).
$$
We will see this below, and also that the matrix $T'(Z+dQ,d)$ corresponding to the multiplication map $\times\partial_{L_Q}$ in the sequence~\eqref{OrigMacDualSeqPrime} defines the same ideal as before.

Since $T'(Z+dQ,d)$ gives a map to the quotient space
$[R^*/(\partial_{L_{P_1}}^d,\ldots,\partial_{L_{P_r}}^d)]_d$,
we need to make some choices to write $T'(Z+dQ,d)$ down explicitly.
After reordering the points $P_i$ if need be,
we may assume $\partial_{L_{P_1}}^d,\ldots,\partial_{L_{P_\alpha}}^d$
give a basis for the vector space 
$\langle\partial_{L_{P_1}}^d,\ldots,\partial_{L_{P_r}}^d\rangle$
spanned by $\partial_{L_{P_1}}^d,\ldots,\partial_{L_{P_r}}^d$,
where $\alpha=\dim \langle\partial_{L_{P_1}}^d,\ldots,\partial_{L_{P_r}}^d\rangle$.
Likewise we may assume that the elements $\partial_{M_j}$ for the degree $d$ monomials $M_1,\ldots,M_\beta$, together with
$\partial_{L_{P_1}}^d,\ldots,\partial_{L_{P_\alpha}}^d$, give a basis for
$[R^*]_d$, where $\beta=\binom{d+n}{n}-\alpha$. 

We now want to write down $T'(Z+dQ,d)$ in terms of the monomial basis $\partial_{m_i}$
for $[R^*]_{d-1}$ and, for  $[R^*/(\partial_{L_{P_1}}^d,\ldots,\partial_{L_{P_r}}^d)]_d$,
the basis $\partial_{M_1},\ldots,\partial_{M_\beta}$ modulo  
$\langle\partial_{L_{P_1}}^d,\ldots,\partial_{L_{P_r}}^d\rangle$.
To do so we first write down $T(Z+dQ,d)$ in terms of the basis 
$\partial_{m_i}$ for $[R^*]_{d-1}$ and the basis 
$$\partial_{L_{P_1}}^d,\ldots,\partial_{L_{P_\alpha}}^d,\partial_{M_1},\ldots,\partial_{M_\beta}$$ 
for $[R^*]_d$. This is just $S^{-1}T(Z+dQ,d)$, where $S$ is the matrix
whose columns are 
$$\partial_{L_{P_1}}^d,\ldots,\partial_{L_{P_\alpha}}^d,\partial_{M_1},\ldots,\partial_{M_\beta}$$
written in terms of the original monomial basis $\partial_{M_i}$.

Thus $T(Z+dQ,d)$ in terms of the basis
$\partial_{L_{P_1}}^d,\ldots,\partial_{L_{P_\alpha}}^d,\partial_{M_1},\ldots,\partial_{M_\beta}$
is the block matrix
\begin{equation}\label{BlockMatrix}
S^{-1}T=\left(
\begin{array}{ccccc}
id_\alpha & | & * & | & **\\
\hline
0 &| & 0 & | & T'(Z+dQ,d)\\
\end{array}
\right),
\end{equation}
where $id_\alpha$
is the $\alpha\times\alpha$ identity matrix, $T'(Z+dQ,d)$ is the 
$\beta\times(\binom{d+n-1}{n})$ matrix representing the map $\times\partial_{L_Q}$
in sequence \eqref{OrigMacDualSeqPrime}, $*$ is an $\alpha\times(r-\alpha)$ matrix and 
$**$ is an $\alpha\times(\binom{n+d-1}{n})$ matrix, but the entries
of $*$ and $**$ will not really concern us. 

Since $T$ and $S^{-1}T$ represent the same map, just with respect to different bases,
we have $I_q(T)=I_q(S^{-1}T)$ for the ideals of $q\times q$ minors for all $q$.
By column reduction (i.e., row reduction applied to the transpose), we obtain 
from $S^{-1}T$ the matrix
$$S^{-1}TU=\left(
\begin{array}{ccccc}
id_\alpha & | & 0 & | & 0\\
\hline
0 &| & 0 & | & T'(Z+dQ,d)\\
\end{array}
\right)
,$$
so (suppressing $Z$ and $dQ$) we also have $I_q(T)=I_q(S^{-1}T)=I_q(S^{-1}TU)$ for every $q$.

For $q\leq \alpha$, these are the unit ideal $(1)$.
For $q>\alpha$, we have $I_q(S^{-1}TU) = I_{q-\alpha}(T')+\cdots+ I_q(T')$.
If for $q=\alpha+1$ we have $I_q(T)=I_q(S^{-1}TU)=(0)$, then 
$T'$ is the zero matrix, hence $\partial_{mL_P}$ is
in $\langle \partial_{L_{P_1}}^d,\ldots,\partial_{L_{P_r}}^d\rangle$
for all points $P$. This happens if and only if
$\langle \partial_{L_{P_1}}^d,\ldots,\partial_{L_{P_r}}^d\rangle=[R^*]_d$,
which in turn means $[I(P_1)\cap\cdots \cap I(P_r)]_d=(0)$ and hence the $d$-Weddle locus 
is empty. Suppose $T'$ is not the 0 matrix. Then for some $q>\alpha$ we have
$I_{q-\alpha}(T')\neq0$ but $I_{q-\alpha+1}(T')=(0)$, hence
$I_{q+1}(T)=I_{q+1}(S^{-1}TU)=I_{q-\alpha+1}(T')=(0)$,
so the zero locus defined by each of the ideals 
$I_q(T)=I_q(T')=I_q(S^{-1}TU)=I_{q-\alpha}(T')\neq0$ is the 
$d$-Weddle locus.

In fact we now will show that $I_q(T(Z+dQ,d))=I_q(\Lambda(Z+dQ,d))$
for all $q$, so the matrices $\Lambda(Z+dQ,d)$ and $T(Z+dQ,d)$ define the same $d$-Weddle ideal and so can be used to
obtain the $d$-Weddle scheme.

In preparation for the proof that
$I_q(T(Z+dQ,d))=I_q(\Lambda(Z+dQ,d))$, 
we make some observations.
First, note that $T(Z+dQ,d)$ and the transpose $N=(\Lambda(Z+dQ,d))^t$
of $\Lambda(Z+dQ,d)$ 
have the same size; both are $\binom{n+d}{n}\times(r+\binom{n+d-1}{n})$
matrices. For both, the entries of the first $r$ columns are 
scalars and the entries of the remaining $\binom{n+d-1}{n}$ 
columns are scalar multiples of the variables $x_i$.

We can compare the entries of the two matrices.
Recall $T_1$ is the submatrix of $T=T(Z+dQ,d)$ consisting of the first $r$
columns of $T$, and $T_2$ is the submatrix 
consisting of the remaining columns of $T$.
Likewise, let $N_1$ be the submatrix of $N$ 
consisting of the first $r$ columns of $N$, and 
let $N_2$ be the submatrix consisting of the remaining columns of $N$.

The entry $(N_1)_{ij}$ for row $i$ and column $j$ of $N_1$ is
the value $M_i(P_j)$ of the monomial $M_i$ at $P_j$.
The entry $(T_1)_{ij}$ is $c_{M_i}M_i(P_j)=d!M_i(P_j)/e_{M_i}$.
The entry $(N_2)_{ij}$ for row $i$ and column $j$ of $N_2$ is
$\partial_{m_j}M_i(Q)$. This is 0 if $m_j\not| M_i$ and it is
$e_{M_i}x_{k_{ij}}$ if $m_j|M_i$ where $x_{k_{ij}}=M_i/m_j$.
The entry $(T_2)_{ij}$ is 0 if $m_j\not| M_i$ and it is
$x_{k_{ij}}$ if $m_j|M_i$ where $x_{k_{ij}}=M_i/m_j$.
Thus 
$$(T_1)_{ij}=c_{M_i}(N_1)_{ij}=d!(N_1)_{ij}/e_{M_i}$$ 
and 
$$(N_2)_{ij}=e_{M_i}(T_2)_{ij}.$$

Since $N$ and $T$ are matrices of the same size,
we can speak of corresponding minors.
A minor of $T$ is the determinant of some
square submatrix $A$ of $T$.
The corresponding minor for $N$ is the determinant
of the matrix $B$ whose entries occupy the same locations in
$N$ as do those of $A$ in $T$.

\begin{proposition}\label{MacDuality=Interp}
Given a finite set of points $Z\subset\PP^n$ and a degree $d$,
let $A$ be a minor of $T=T(Z+dQ,d)$, 
coming from a given choice of $s$ rows and $s$ columns of $T$.
Assume that the rows correspond to $\partial_{M_{i_j}}$ for monomials $M_{i_1},\ldots,M_{i_s}$, and that
$j$ of the chosen columns come from $T_1$.
Let $B$ be the corresponding minor of $N=(\Lambda(Z+dQ,d))^t$. 
Then 
$$B=\frac{e_{M_{i_1}}\cdots e_{M_{i_s}}}{(d!)^j}A$$
and thus $I_s(T(Z+dQ,d))=I_s(\Lambda(Z+dQ,d))$.
\end{proposition}

\begin{proof}
The minor $A$ is a signed sum of products of $s$ entries of $T$,
where no two of the $s$ entries come from the same row or column.
Let $\pi$ be one of these products.
Let $j$ be the number of entries in $\pi$ which come from $T_1$; let's say those entries 
were chosen from rows corresponding to $M_{i_1},\ldots, M_{i_j}$.
The other entries come from $T_2$ in rows corresponding to 
$M_{i_{j+1}},\ldots, M_{i_s}$.

The corresponding product for $N$ is 
$$\frac{e_{M_{i_{j+1}}}\cdots e_{M_{i_s}}}{c_{M_{i_1}}\cdots c_{M_{i_j}}}\pi=
\frac{e_{M_{i_{j+1}}}\cdots e_{M_{i_s}}}{\frac{(d!)^j}{e_{M_{i_1}}\cdots e_{M_{i_j}}}}\pi=
\frac{e_{M_{i_1}}\cdots e_{M_{i_s}}}{(d!)^j}\pi$$
using the notation introduced in display \eqref{notation1}.
Thus all terms in the minor $A$ are multiplied by the same factor
$\frac{e_{M_{i_1}}\cdots e_{M_{i_s}}}{(d!)^j}$ to get the terms for $B$.
\end{proof}

\begin{example}\label{BigWeddleExample}
Consider the case of $n=3$ and degree $d=3$
with $Z$ consisting of the points 
$$\begin{array}{ccccc}
P_1=[1:0:0:0], && P_2=[0:1:0:0], && P_3=[0:0:1:0],\\
P_4=[0:0:0:1], && P_5=[1:1:1:1], && P_6=[2:3:5:7].
\end{array}$$
We will give the matrices $N=(\Lambda(Z+3Q,3))^t$ and $T=T(Z+3Q,3)$ with respect to the following
bases of $[R]_d$ and $[R]_{d-1}$.
For $[R]_3$ we have
$$
\begin{array}{lllllllll}
M_1=x_0^3, && 
M_2=x_0^2x_1, &&
M_3=x_0^2x_2, &&
M_4=x_0^2x_3, &&
M_5=x_0x_1^2\\
M_6=x_0x_1x_2, &&
M_7=x_0x_1x_3, &&
M_8=x_0x_2^2, &&
M_9=x_0x_2x_3, &&
M_{10}=x_0x_3^2, \\
M_{11}=x_1^3, &&
M_{12}=x_1^2x_2, &&
M_{13}=x_1^2x_3, &&
M_{14}=x_1x_2^2, &&
M_{15}=x_1x_2x_3, \\
M_{16}=x_1x_3^2, &&
M_{17}=x_2^3, &&
M_{18}=x_2^2x_3, && 
M_{19}=x_2x_3^2, &&
M_{20}=x_3^3,
\end{array}
$$
and
for $[R]_2$ we have 
$$
\begin{array}{lllllllll}
m_1=x_0^2, && 
m_2=x_0x_1, &&
m_3=x_0x_2, &&
m_4=x_0x_3, &&
m_5=x_1^2, \\
m_6=x_1x_2, &&
m_7=x_1x_3, &&
m_8=x_2^2, &&
m_9=x_2x_3, &&
m_{10}=x_3^2.
\end{array}
$$

For $N$, row $i$ corresponds to the basis element $M_i$,
column $j$, for $1\leq j\leq 6$, corresponds to the point $P_j$,
column $j$, for $6< j\leq 16$, corresponds to the basis element
$m_{j-6}$.
Entry $N_{i,j}$, for $1\leq j\leq 6$, is $M_i(P_j)$, while for
$6< j\leq 16$, it is
$\partial_{m_{j-6}}M_i$.
So for example
$N_{2,6}=M_2(P_6)=x_0^2x_1([2:3:5:7])=2^23=12$, and
$N_{11,11}=\partial_{m_5}M_{11}=\partial_{x_1^2}(x_1^3)=6x_1$.
Here is the entire matrix $N$ (with $N_{2,6}$ and $N_{11,11}$ shown boxed):
{\tiny$$
\left(
\begin{array}{cccccccccccccccc}
1 & 0 & 0 & 0 & 1 & 8 & 6x_0 & 0 & 0 & 0 & 0 & 0 & 0 & 0 & 0 & 0\\
0 & 0 & 0 & 0 & 1 & \fbox{12} & 2x_1 & 2x_0 & 0 & 0 & 0 & 0 & 0 & 0 & 0 & 0\\
0 & 0 & 0 & 0 & 1 & 20 & 2x_2 & 0 & 2x_0 & 0 & 0 & 0 & 0 & 0 & 0 & 0\\
0 & 0 & 0 & 0 & 1 & 28 & 2x_3 & 0 & 0 & 2x_0 & 0 & 0 & 0 & 0 & 0 & 0\\
0 & 0 & 0 & 0 & 1 & 18 & 0 & 2x_1 & 0 & 0 & 2x_0 & 0 & 0 & 0 & 0 & 0\\
0 & 0 & 0 & 0 & 1 & 30 & 0 & x_2 & x_1 & 0 & 0 & x_0 & 0 & 0 & 0 & 0\\
0 & 0 & 0 & 0 & 1 & 42 & 0 & x_3 & 0 & x_1 & 0 & 0 & x_0 & 0 & 0 & 0\\
0 & 0 & 0 & 0 & 1 & 50 & 0 & 0 & 2x_2 & 0 & 0 & 0 & 0 & 2x_0 & 0 & 0\\
0 & 0 & 0 & 0 & 1 & 70 & 0 & 0 & x_3 & x_2 & 0 & 0 & 0 & 0 & x_0 & 0\\
0 & 0 & 0 & 0 & 1 & 98 & 0 & 0 & 0 & 2x_3 & 0 & 0 & 0 & 0 & 0 & 2x_0\\
0 & 1 & 0 & 0 & 1 & 27 & 0 & 0 & 0 & 0 & \fbox{$6x_1$} & 0 & 0 & 0 & 0 & 0\\
0 & 0 & 0 & 0 & 1 & 45 & 0 & 0 & 0 & 0 & 2x_2 & 2x_1 & 0 & 0 & 0 & 0\\
0 & 0 & 0 & 0 & 1 & 63 & 0 & 0 & 0 & 0 & 2x_3 & 0 & 2x_1 & 0 & 0 & 0\\
0 & 0 & 0 & 0 & 1 & 75 & 0 & 0 & 0 & 0 & 0 & 2x_2 & 0 & 2x_1 & 0 & 0\\
0 & 0 & 0 & 0 & 1 & 105 & 0 & 0 & 0 & 0 & 0 & x_3 & x_2 & 0 & x_1 & 0\\
0 & 0 & 0 & 0 & 1 & 147 & 0 & 0 & 0 & 0 & 0 & 0 & 2x_3 & 0 & 0 & 2x_1\\
0 & 0 & 1 & 0 & 1 & 125 & 0 & 0 & 0 & 0 & 0 & 0 & 0 & 6x_2 & 0 & 0\\
0 & 0 & 0 & 0 & 1 & 175 & 0 & 0 & 0 & 0 & 0 & 0 & 0 & 2x_3 & 2x_2 & 0\\
0 & 0 & 0 & 0 & 1 & 245 & 0 & 0 & 0 & 0 & 0 & 0 & 0 & 0 & 2x_3 & 2x_2\\
0 & 0 & 0 & 1 & 1 & 343 & 0 & 0 & 0 & 0 & 0 & 0 & 0 & 0 & 0 & 6x_3\\
\end{array}
\right).$$}
And here for reference is the transpose of $4\Lambda'({Z+3Q,3})$:
{\tiny$$
\left(
\begin{array}{cccccccccc}
8x_1-16x_2+8x_3 &  8x_0 &  -16x_0 &  8x_0 &  0 &  0 &  0 &  0 &  0 &  0\\
-2x_1-6x_2 &  -2x_0+8x_1 &  -6x_0 &  0 &  8x_0 &  0 &  0 &  0 &  0 &  0\\
10x_1-18x_2 &  10x_0+4x_2 &  -18x_0+4x_1 &  0 &  0 &  4x_0 &  0 &  0 &  0 &  0\\
22x_1-30x_2 &  22x_0+4x_3 &  -30x_0 &  4x_1 &  0 &  0 &  4x_0 &  0 &  0 &  0\\
30x_1-38x_2 &  30x_0 &  -38x_0+8x_2 &  0 &  0 &  0 &  0 &  8x_0 &  0 &  0\\
50x_1-58x_2 &  50x_0 &  -58x_0+4x_3 &  4x_2 &  0 &  0 &  0 &  0 &  4x_0 &  0\\
78x_1-86x_2 &  78x_0 &  -86x_0 &  8x_3 &  0 &  0 &  0 &  0 &  0 &  8x_0\\
25x_1-33x_2 &  25x_0 &  -33x_0 &  0 &  8x_2 &  8x_1 &  0 &  0 &  0 &  0\\
43x_1-51x_2 &  43x_0 &  -51x_0 &  0 &  8x_3 &  0 &  8x_1 &  0 &  0 &  0\\
55x_1-63x_2 &  55x_0 &  -63x_0 &  0 &  0 &  8x_2 &  0 &  8x_1 &  0 &  0\\
85x_1-93x_2 &  85x_0 &  -93x_0 &  0 &  0 &  4x_3 &  4x_2 &  0 &  4x_1 &  0\\
127x_1-135x_2 &  127x_0 &  -135x_0 &  0 &  0 &  0 &  8x_3 &  0 &  0 &  8x_1\\
155x_1-163x_2 &  155x_0 &  -163x_0 &  0 &  0 &  0 &  0 &  8x_3 &  8x_2 &  0\\
225x_1-233x_2 &  225x_0 &  -233x_0 &  0 &  0 &  0 &  0 &  0 &  8x_3 &  8x_2\\\end{array}
\right).$$}

For $T$, we have 
$T_{i,j}=c_{M_i}M_i(P_j)$ 
for $1\leq j\leq 6$, and
$T_{i,j}=0$ if $m_{j-6}$ does not divide $M_i$, and it is
$T_{i,j}=M_i/m_{j-6}$ if $m_{j-6} | M_i$.
So for example
$$T_{2,6}=c_{M_2}M_2(P_6)=\frac{3!}{2! 1!}(x_0^2x_1)([2:3:5:7])=36,$$
and $T_{11,11}=M_{11}/m_5=x_1^3/(x_1^2)=x_1$.
Here is the entire matrix (with $T_{2,6}$ and $T_{11,11}$ shown boxed):
{\tiny$$
T=\left(
\begin{array}{cccccccccccccccc}
1 & 0 & 0 & 0 & 1 & 8 & x_0 & 0 & 0 & 0 & 0 & 0 & 0 & 0 & 0 & 0\\ 
0 & 0 & 0 & 0 & 3 & \fbox{36} & x_1 & x_0 & 0 & 0 & 0 & 0 & 0 & 0 & 0 & 0\\ 
0 & 0 & 0 & 0 & 3 & 60 & x_2 & 0 & x_0 & 0 & 0 & 0 & 0 & 0 & 0 & 0\\ 
0 & 0 & 0 & 0 & 3 & 84 & x_3 & 0 & 0 & x_0 & 0 & 0 & 0 & 0 & 0 & 0\\ 
0 & 0 & 0 & 0 & 3 & 54 & 0 & x_1 & 0 & 0 & x_0 & 0 & 0 & 0 & 0 & 0\\ 
0 & 0 & 0 & 0 & 6 & 180 & 0 & x_2 & x_1 & 0 & 0 & x_0 & 0 & 0 & 0 & 0\\ 
0 & 0 & 0 & 0 & 6 & 252 & 0 & x_3 & 0 & x_1 & 0 & 0 & x_0 & 0 & 0 & 0\\ 
0 & 0 & 0 & 0 & 3 & 150 & 0 & 0 & x_2 & 0 & 0 & 0 & 0 & x_0 & 0 & 0\\ 
0 & 0 & 0 & 0 & 6 & 420 & 0 & 0 & x_3 & x_2 & 0 & 0 & 0 & 0 & x_0 & 0\\ 
0 & 0 & 0 & 0 & 3 & 294 & 0 & 0 & 0 & x_3 & 0 & 0 & 0 & 0 & 0 & x_0\\ 
0 & 1 & 0 & 0 & 1 & 27 & 0 & 0 & 0 & 0 & \fbox{$x_1$} & 0 & 0 & 0 & 0 & 0\\ 
0 & 0 & 0 & 0 & 3 & 135 & 0 & 0 & 0 & 0 & x_2 & x_1 & 0 & 0 & 0 & 0\\ 
0 & 0 & 0 & 0 & 3 & 189 & 0 & 0 & 0 & 0 & x_3 & 0 & x_1 & 0 & 0 & 0\\ 
0 & 0 & 0 & 0 & 3 & 225 & 0 & 0 & 0 & 0 & 0 & x_2 & 0 & x_1 & 0 & 0\\ 
0 & 0 & 0 & 0 & 6 & 630 & 0 & 0 & 0 & 0 & 0 & x_3 & x_2 & 0 & x_1 & 0\\ 
0 & 0 & 0 & 0 & 3 & 441 & 0 & 0 & 0 & 0 & 0 & 0 & x_3 & 0 & 0 & x_1\\ 
0 & 0 & 1 & 0 & 1 & 125 & 0 & 0 & 0 & 0 & 0 & 0 & 0 & x_2 & 0 & 0\\ 
0 & 0 & 0 & 0 & 3 & 525 & 0 & 0 & 0 & 0 & 0 & 0 & 0 & x_3 & x_2 & 0\\ 
0 & 0 & 0 & 0 & 3 & 735 & 0 & 0 & 0 & 0 & 0 & 0 & 0 & 0 & x_3 & x_2\\ 
0 & 0 & 0 & 1 & 1 & 343 & 0 & 0 & 0 & 0 & 0 & 0 & 0 & 0 & 0 & x_3\\
\end{array}
\right).$$}
If we take the $16\times16$ submatrix obtained from $T$ by deleting
rows 2, 12, 18, 19, its determinant is some nonzero polynomial $A$ of degree 10; 
this is a minor of $T$. By direct calculation we find that the minor $B$ of $N$ 
obtained by deleting the same rows is 
$$B=(64/9)A=\frac{1}{6^6}\left(\prod_{\substack{i=1,\ldots,20\\ i\neq 2,12,18,19}}e_{M_i} \right)A,$$
exactly as asserted by Proposition \ref{MacDuality=Interp}.

To find $T'$, we need to find the change of basis matrix $S$.
The first $\alpha$ columns of $S$ should be a basis for the column space of the
first $r$ columns of $T$; here $\alpha=r=6$, so we just take the first $r$ columns 
of $T$ for the first $r$ columns of $S$. The remaining $\binom{d+n}{n}-r=14$
columns of $S$ should be unit vectors making $S$ invertible.
It is not hard by inspection to see that the following choice of unit vector
columns gives an invertible matrix:
{\tiny$$
S=
\left(
\begin{array}{cccccccccccccccccccc}
1 & 0 & 0 & 0 & 1 & 8 & 0 & 0 & 0 & 0 & 0 & 0 & 0 & 0 & 0 & 0 & 0 & 0 & 0 & 0\\
0 & 0 & 0 & 0 & 3 & 36 & 0 & 0 & 0 & 0 & 0 & 0 & 0 & 0 & 0 & 0 & 0 & 0 & 0 & 0\\
0 & 0 & 0 & 0 & 3 & 60 & 0 & 0 & 0 & 0 & 0 & 0 & 0 & 0 & 0 & 0 & 0 & 0 & 0 & 0\\
0 & 0 & 0 & 0 & 3 & 84 & 1 & 0 & 0 & 0 & 0 & 0 & 0 & 0 & 0 & 0 & 0 & 0 & 0 & 0\\
0 & 0 & 0 & 0 & 3 & 54 & 0 & 1 & 0 & 0 & 0 & 0 & 0 & 0 & 0 & 0 & 0 & 0 & 0 & 0\\
0 & 0 & 0 & 0 & 6 & 180 & 0 & 0 & 1 & 0 & 0 & 0 & 0 & 0 & 0 & 0 & 0 & 0 & 0 & 0\\
0 & 0 & 0 & 0 & 6 & 252 & 0 & 0 & 0 & 1 & 0 & 0 & 0 & 0 & 0 & 0 & 0 & 0 & 0 & 0\\
0 & 0 & 0 & 0 & 3 & 150 & 0 & 0 & 0 & 0 & 1 & 0 & 0 & 0 & 0 & 0 & 0 & 0 & 0 & 0\\
0 & 0 & 0 & 0 & 6 & 420 & 0 & 0 & 0 & 0 & 0 & 1 & 0 & 0 & 0 & 0 & 0 & 0 & 0 & 0\\
0 & 0 & 0 & 0 & 3 & 294 & 0 & 0 & 0 & 0 & 0 & 0 & 1 & 0 & 0 & 0 & 0 & 0 & 0 & 0\\
0 & 1 & 0 & 0 & 1 & 27 & 0 & 0 & 0 & 0 & 0 & 0 & 0 & 0 & 0 & 0 & 0 & 0 & 0 & 0\\
0 & 0 & 0 & 0 & 3 & 135 & 0 & 0 & 0 & 0 & 0 & 0 & 0 & 1 & 0 & 0 & 0 & 0 & 0 & 0\\
0 & 0 & 0 & 0 & 3 & 189 & 0 & 0 & 0 & 0 & 0 & 0 & 0 & 0 & 1 & 0 & 0 & 0 & 0 & 0\\
0 & 0 & 0 & 0 & 3 & 225 & 0 & 0 & 0 & 0 & 0 & 0 & 0 & 0 & 0 & 1 & 0 & 0 & 0 & 0\\
0 & 0 & 0 & 0 & 6 & 630 & 0 & 0 & 0 & 0 & 0 & 0 & 0 & 0 & 0 & 0 & 1 & 0 & 0 & 0\\
0 & 0 & 0 & 0 & 3 & 441 & 0 & 0 & 0 & 0 & 0 & 0 & 0 & 0 & 0 & 0 & 0 & 1 & 0 & 0\\
0 & 0 & 1 & 0 & 1 & 125 & 0 & 0 & 0 & 0 & 0 & 0 & 0 & 0 & 0 & 0 & 0 & 0 & 0 & 0\\
0 & 0 & 0 & 0 & 3 & 525 & 0 & 0 & 0 & 0 & 0 & 0 & 0 & 0 & 0 & 0 & 0 & 0 & 1 & 0\\
0 & 0 & 0 & 0 & 3 & 735 & 0 & 0 & 0 & 0 & 0 & 0 & 0 & 0 & 0 & 0 & 0 & 0 & 0 & 1\\
0 & 0 & 0 & 1 & 1 & 343 & 0 & 0 & 0 & 0 & 0 & 0 & 0 & 0 & 0 & 0 & 0 & 0 & 0 & 0\\
\end{array}
\right).$$}
Now $S^{-1}T$ is the block matrix shown in equation \eqref{BlockMatrix} with
$T'$ being the block in the lower right corner. What we get for $T'$ is
{\tiny$$T'=
\left(\begin{array}{cccccccccc}
x_1-2x_2+x_3 & x_0 & -2x_0 & x_0 & 0 & 0 & 0 & 0 & 0 & 0\\
-\frac{1}{4}x_1-\frac{3}{4}x_2 & -\frac{1}{4}x_0+x_1 & -\frac{3}{4}x_0 & 0 & x_0 & 0 & 0 & 0 & 0 & 0\\
\frac{5}{2}x_1-\frac{9}{2}x_2 & \frac{5}{2}x_0+x_2 & -\frac{9}{2}x_0+x_1 & 0 & 0 & x_0 & 0 & 0 & 0 & 0\\
\frac{11}{2}x_1-\frac{15}{2}x_2 & \frac{11}{2}x_0+x_3 & -\frac{15}{2}x_0 & x_1 & 0 & 0 & x_0 & 0 & 0 & 0\\
\frac{15}{4}x_1-\frac{19}{4}x_2 & \frac{15}{4}x_0 & -\frac{19}{4}x_0+x_2 & 0 & 0 & 0 & 0 & x_0 & 0 & 0\\
\frac{25}{2}x_1-\frac{29}{2}x_2 & \frac{25}{2}x_0 & -\frac{29}{2}x_0+x_3 & x_2 & 0 & 0 & 0 & 0 & x_0 & 0\\
\frac{39}{4}x_1-\frac{43}{4}x_2 & \frac{39}{4}x_0 & -\frac{43}{4}x_0 & x_3 & 0 & 0 & 0 & 0 & 0 & x_0\\
\frac{25}{8}x_1-\frac{33}{8}x_2 & \frac{25}{8}x_0 & -\frac{33}{8}x_0 & 0 & x_2 & x_1 & 0 & 0 & 0 & 0\\
\frac{43}{8}x_1-\frac{51}{8}x_2 & \frac{43}{8}x_0 & -\frac{51}{8}x_0 & 0 & x_3 & 0 & x_1 & 0 & 0 & 0\\
\frac{55}{8}x_1-\frac{63}{8}x_2 & \frac{55}{8}x_0 & -\frac{63}{8}x_0 & 0 & 0 & x_2 & 0 & x_1 & 0 & 0\\
\frac{85}{4}x_1-\frac{93}{4}x_2 & \frac{85}{4}x_0 & -\frac{93}{4}x_0 & 0 & 0 & x_3 & x_2 & 0 & x_1 & 0\\
\frac{127}{8}x_1-\frac{135}{8}x_2 & \frac{127}{8}x_0 & -\frac{135}{8}x_0 & 0 & 0 & 0 & x_3 & 0 & 0 & x_1\\
\frac{155}{8}x_1-\frac{163}{8}x_2 & \frac{155}{8}x_0 & -\frac{163}{8}x_0 & 0 & 0 & 0 & 0 & x_3 & x_2 & 0\\
\frac{225}{8}x_1-\frac{233}{8}x_2 & \frac{225}{8}x_0 & -\frac{233}{8}x_0 & 0 & 0 & 0 & 0 & 0 & x_3 & x_2\\
\end{array}
\right).$$}
One can check by direct computation that the ideals of maximal minors for the matrices
$N,T,T'$ and $\Lambda'({Z+3Q,3})$ are indeed equal and nonzero (indeed, the rows of $T'$ are nonzero scalar multiples of the rows of
$(\Lambda'({Z+3Q,3}))^t$). 
Thus any of these matrices can be used
to find the 3-Weddle scheme. It is interesting to mention that by direct computation
we find the 3-Weddle ideal is not saturated, and its saturation is not radical,
so the 3-Weddle scheme is not reduced and thus is not equal to the 3-Weddle locus.
In fact, the scheme defined by the 3-Weddle ideal consists of the union of the 15 lines joining pairs of the  points $P_1,\dots,P_6$, together with embedded components at each point $P_i$. More precisely, a primary decomposition for the ideal of the 3-Weddle scheme is given by the intersection of the ideals of the 15 lines with the cubes of the ideals of the six points.
\end{example}

\section{A first look to the \texorpdfstring{$d$}{d}-Weddle scheme of general points in \texorpdfstring{$\PP^n$}{Pn}}\label{sec: general points in Pn}
The results introduced in Section \ref{sec:2approaches} allow us to start a systematic study of the $d$-Weddle scheme for sets of general points in $\mathbb P^n$. We start with an illustrative example. 

\begin{example}[The $2$-Weddle scheme of $10$ general points in $\PP^4$]\label{10ptsInP4}
Let $X$ be a set of 10 general points in $\PP^4$. For a general point $P$, the projection of $X$ is a general set of 10 points in
$\PP^3$, which therefore do not lie on a quadric. Thus it is of interest to investigate the 2-Weddle locus (and scheme) for $X$, that is the locus of points $P$ in $\PP^4$ so that projecting $X$ from $P$ gives a set of 10 points on a quadric surface in $\PP^3$. 

We have
\[
\dim [I(P_1) \cap \dots \cap I(P_{10})]_1 = 0, \ \ \dim [I(P_1) \cap \dots \cap I(P_{10})]_2 = 15 - 10 = 5.
\]
We want to know for which $P$ is it true that
\[
 \dim[I(P_1) \cap \dots \cap I(P_{10}) \cap  I(P)^2]_2 \geq 1.
\]
Since for any distinct points $P_i$ there is a canonical vector space isomorphism
$$[R^*/((\partial_{L_{P_1}})^{t-k_1+1},\ldots,(\partial_{L_{P_r}})^{t-k_r+1})]_t
\cong [R/((L_{P_1})^{t-k_1+1},\ldots,(L_{P_r})^{t-k_r+1})]_t,$$
we can apply Macaulay duality while working in $R$, which is notationally
simpler.
So, by Macaulay duality for any $t$ we have
\[
\dim [I(P_1) \cap \dots \cap I(P_{10})]_t = \dim [R/(L_{P_1}^t, \dots, L_{P_{10}}^t)]_t 
\]
and 
\[ 
\dim [I(P_1) \cap \dots \cap I(P_{10}) \cap I(P)^2]_2 = \dim [R/(L_{P_1}^2, \dots, L_{P_{10}}^2, L_P)]_2.
\]
In the exact sequence
\[
[R/(L_{P_1}^2, \dots, L_{P_{10}}^2)]_1 \stackrel{\times L_P}{\longrightarrow} [R/(L_{P_1}^2, \dots, L_{P_{10}}^2)]_2 \rightarrow [R/(L_{P_1}^2, \dots, L_{P_{10}}^2, L_P)]_2 \rightarrow 0,
\]
the first two terms both have dimension 5. Since for most points $P$ in $\PP^4$ the 
projection from $P$ to $\PP^3$ does not lie on a quadric surface, 
for most linear forms $L_P$ the map 
by $\times L_P$ in the above sequence will be surjective. 
Thus the ideal $I_2(T'(X+2Q,2))$ is not $(0)$,
so the 2-Weddle locus is the degeneracy locus given by $I_2(T'(X+2Q,2))$, and since
$T'(X+2Q,2)$ is a $5\times 5$ matrix, $I_2(T'(X+2Q,2))=(\det(T'(X+2Q,2)))$.
Since $T'(X+2Q,2)$ is a matrix of linear forms, we see the 2-Weddle scheme
is a hypersurface of degree 5. 
\end{example}

More generally, we have the following result.

\begin{theorem}\label{WeddleSchemeSigmaThm}
Fix a hyperplane $\PP^n$ inside $\PP^{n+1}$. Let $Z$ be a general set of $\binom{d+n}{n}$ points in $\PP^{n+1}$. For a point $P \in \PP^{n+1}$ 
let $\pi_P$ be the projection to $\PP^n$. 
Let $\Sigma$ be the $d$-Weddle scheme of points in $\PP^{n+1}$ for which $\pi_P(Z)$ lies on a hypersurface in $\PP^n$ of degree $d$. Then $\Sigma$ is defined by a form
of degree $\binom{d+n}{n+1}$.
\end{theorem}

\begin{proof}
We keep the same set-up and notation as before. 
Let $r=\binom{d+n}{n}$, and let $Z$ consist of the points
$P_1,\ldots,P_r$. We begin with the exact sequence
$$
[R/(L_{P_1}^d, \dots, L_{P_r}^d)]_{d-1} \stackrel{\times L_P}{\longrightarrow} [R/(L_{P_1}^d, \dots, L_{P_r}^d)]_d \rightarrow [R/(L_{P_1}^d,\dots,L_{P_r}^d, L_P)]_d \rightarrow 0.
$$
The first vector space has dimension $\binom{d-1+n+1}{n+1} = \binom{d+n}{n+1}$. The second has dimension
$$
\binom{d+n+1}{n+1} - r = \binom{d+n+1}{n+1} - \binom{d+n}{n} = \binom{d+n}{n+1}.
$$
Thus the multiplication is represented by a square $\displaystyle\binom{d+n}{n+1} \times \binom{d+n}{n+1}$ matrix of linear forms.
Note, as in the previous example, the determinant of this matrix is not 0, 
so the result follows as before. 
\end{proof}

\begin{conjecture}\label{conj:irreducible}
Let $\Sigma$ be as in Theorem \ref{WeddleSchemeSigmaThm}.
Then $\Sigma$ is reduced and irreducible.
\end{conjecture}

Since integrality is a Zariski--open condition,  to prove the conjecture for a given $d$ it would suffice to exhibit a single configuration for which the $d$-Weddle scheme is reduced and irreducible; irreducibility for general $Z$ would then follow.
Alternatively, if one can show the singular locus of $\Sigma$ has codimension in $\Sigma$ more than 1,
then $\Sigma$ is irreducible. This raises the question of what the singular locus of $\Sigma$ is.
This is addressed in Appendix~\ref{AppendixA} for the classical Weddle surface.

\medskip

In general (as in Example \ref{BigWeddleExample}) 
the $d$-Weddle scheme need not be reduced 
(even, unlike Example \ref{BigWeddleExample}, if it is a hypersurface) or equidimensional.
For examples of these phenomena, see Section \ref{non-reduced} and Example \ref{nonred sup on pt}.
Thus it can happen that the $d$-Weddle locus may have smaller degree than the $d$-Weddle scheme.

In some cases, such as 
Example \ref{10ptsInP4},
Theorem \ref{WeddleSchemeSigmaThm} and Remark \ref{ClassicalWeddleResult},
the value of the degree of the $d$-Weddle scheme is 
a consequence of the following fact
taken from \cite{J1}, which studies a disguised version of a similar problem.
Let $A$ be a $q \times p$ matrix of linear forms in $m+1$ variables, and let $Y_r$ be the subscheme of $\PP^m$ defined by the vanishing of the $(r+1) \times (r+1)$ minors of $A$. If $Y_r$ is not empty, then $Y_r$ has codimension at most $(p-r)(q-r)$ (cf. \cite[page 189]{EN}, \cite[Theorem 2.1]{BV}  or \cite[page 83 and Proposition~4.1]{ACGH}). 
For a generic matrix this is achieved (see \cite{HE} for references) and in this case $Y_r$ is ACM \cite{HE}. We will refer to $(p-r)(q-r)$ as the  {\it expected codimension}. Note that if this maximum codimension is achieved, the degree of $Y_r$ (as a scheme) is forced:

\begin{lemma}\label{J1 lemma}
Assume $Y_r  \neq \emptyset$ has the expected codimension $(p-r)(q-r)$. 
Then by \cite[Lemma 1.4]{J1}
\[
\deg Y_r = \prod_{i=0}^{p-r-1} \left [ \binom{q+i}{r}  /  \binom{r+i}{r} \right ],
\]
and by \cite{HE} $Y_r$ is ACM.
\end{lemma}

\begin{remark}
Sometimes one is interested in situations where 
the codimension is not the expected one, so Lemma \ref{J1 lemma}
does not apply. 
We give such an example in subsection \ref{not max minors},
which also has the point of interest that the minors needed are nonmaximal.
\end{remark}

\section{Weddle schemes and loci for sets of six points in \texorpdfstring{$\PP^3$}{P3}} \label{six pt section}

In this section, we describe several interesting phenomena exhibited by Weddle schemes associated with various configurations of six points in $\mathbb P^3$. It is perhaps surprising how rich a range of behavior can arise from such a small case. 

\subsection{The classical Weddle surface}
We start by applying the Macaulay duality 
method in the case that $Z\subset \PP^3$ consists of 
6 points in LGP, in particular recovering the fact, due to Weddle, 
that the $2$-Weddle scheme (and locus) of a set of 6 points in LGP 
is a surface of degree 4. 
\begin{remark}\label{ClassicalWeddleResult}
Six points in LGP in $\PP^3$ necessarily impose independent conditions on 
forms of degree 2, since five points in LGP can, up to change of coordinates, 
be taken to be the four coordinate vertices and the point $[1:1:1:1]$.
It is easy to see that these five points impose independent conditions on quadrics,
but the ideal of these five points is generated by quadrics, 
so any additional sixth point will impose an additional condition. 

Now, looking at the sequence \eqref{OrigMacDualSeqPrime}
with $d=2$, we see that $\times\partial_{L_P}$ is multiplication by a linear form
from one vector space of dimension 4 to another 
vector space of dimension 4, so $\times\partial_{L_P}$ is represented by a $4\times4$ matrix $J$ of linear forms
in the coordinates of $P$ (namely $J=T'(Z+2P,2)$ to be explicit, as given in Equation \eqref{BlockMatrix}). 
Furthermore, not all projections of $Z$ lie on a conic (for example, project from a 
general point $P$ in the plane of 3 of the points of $Z$, hence $P$ is not in the plane of the other 3 points of $Z$). Thus the determinant of $J$ does not vanish identically,
so the $2$-Weddle scheme is defined by $\det(J)$.
Hence the $2$-Weddle scheme has degree 4 and 
all of its components have codimension 1; i.e., it is a quartic surface. 
To see that the $2$-Weddle locus is also a quartic surface, we claim 
that there are lines which intersect the $2$-Weddle locus in at least 4 points.
Thus the degree of the $2$-Weddle locus is at least 4, but its degree cannot be
more than the degree of the $2$-Weddle scheme, so
the $2$-Weddle locus is also a quartic surface.
To see the claim, let $L$ be a general line in the plane defined by 3 of the points of $Z$.
There are 4 points $P\in L$ from which the projection of $Z$ is contained
in a conic, namely the points of intersection of $L$ with the lines through 
pairs of the 3 points, and the point of intersection of $L$ with the plane defined by the
other three points of $Z$. 
\end{remark}

\subsection{Weddle surfaces with more than 25 lines}	\label{more25 lines}
We recall first that a general Weddle surface contains $25$ lines. To this end
let $Z=\{P_1,\ldots,P_6\}\subseteq \PP^3$ be a set of 6 points in LGP. 
By Remark \ref{ClassicalWeddleResult}, the 2-Weddle surface $\mathcal W(Z)$
is a quartic surface. Since the points of $Z$ are in LGP,
there are 15 distinct lines through pairs of points of $Z$.
In addition, LGP ensures that each subset of three of the six points of $Z$ is 
contained in a unique plane. Thus, given two planes $\Pi_1, \Pi_2$ with $Z\subset \Pi_1\cup\Pi_2$, each intersection $Z\cap \Pi_i$ has exactly three points of $Z$, so
$Z\cap \Pi_1$ is disjoint from $Z\cap \Pi_2$. Thus no point of $Z$ is on the line
$\Pi_1\cap\Pi_2$, and no other plane containing this line contains a point of $Z$.
Thus there are $\frac{1}{2}\binom{6}{3}=10$ distinct lines which arise as intersections of two planes whose union
contains $Z$, and these lines are distinct from the 15 lines 
through pairs of points of $Z$. These 25 lines are contained in $\mathcal W(Z)$
since projecting from a general point $O$ of any such line, the image of $Z$ is 
contained in a conic, and so $O$ (and hence the line) is contained in 
$\mathcal W(Z)$. This is classical -- see \cite{EMCH}.

However it is also known that there can be more than 25 lines in $\mathcal W(Z)$ if the points are not general, but still in LGP.
This case occurs for instance when the points in $Z$ are pairs of a skew involution, see \cite[Theorem 1]{moore}. 
For the reader's convenience we include a statement and a proof of such a result, namely
Proposition \ref{p.Weddle >25}, which gives a construction having two lines
in addition to the 25 mentioned above, and which shows precisely
how to pick the points so that they are in LGP.

The following lemma will be useful in the proof of Proposition \ref{p.Weddle >25}.
It is a corollary of Desargues' involution theorem (cf.~\cite[Theorem~9.31]{CoxeterPG}).

\begin{lemma}\label{l. harmonic and conic}
Let $A\in\PP^2$ and let $\ell_1,\ell_2,\ell_3$ be three distinct lines through $A$.
For $i=1,2,3$, choose points $B_i,C_i\in \ell_i\setminus\{A\}$ and let $H_i\in\ell_i$ be the harmonic conjugate of $A$
with respect to $B_i,C_i$, i.e., $(A,H_i;B_i,C_i)$ is harmonic.
Then the following are equivalent:
\begin{enumerate}
\item[(i)] there exists a (possibly reducible) conic $\gamma$ containing $B_1,B_2,B_3,C_1,C_2,C_3$;
\item[(ii)] the three points $H_1,H_2,H_3$ are collinear.
\end{enumerate}
\end{lemma}

\begin{remark}
    In the situation of Lemma \ref{l. harmonic and conic}, if $\gamma$ is irreducible then the line $\ell=\overline{H_1H_2H_3}$ is the polar of $A$ with respect to $\gamma$,
and if $\gamma$ is reducible then $\ell$ passes through the singular point of $\gamma$.
Moreover, the nine points $H_i,B_i,C_i$ are the complete intersection
$(\ell_1\cup\ell_2\cup\ell_3)\cap(\ell\cup\gamma)$.
\end{remark}

\begin{proposition}\label{p.Weddle >25}
Let $\ell_1,\ell_2, \ell_3$ be three skew lines in $\PP^3$, contained in the smooth quadric $\mathcal Q$.
Let  $m_1, m_2$ be distinct lines intersecting each of $\ell_1,\ell_2, \ell_3$. 
Let $P_{ij}$ be the point of intersection of $\ell_i$ and $m_j$. 
For each $i=1,2,3$, 
pick any distinct points $A_i, B_i$ on $\ell_i$ distinct from the points $P_{ij}$.
Let
$Z=\{A_1,B_1,A_2,B_2,A_3,B_3\}$. Furthermore,
let $r_i$ be the ruling line of $\mathcal Q$ through $A_i$ other than $\ell_i$ and let 
$r'_i$ be the ruling line of $\mathcal Q$ through $B_i$ other than $\ell_i$.
\begin{enumerate}
\item[(1)] The following are equivalent:
\begin{enumerate}
\item[(a)] the points of $Z$ are in LGP;
\item[(b)] $A_2, B_2\not\in r_1\cup r'_1$ and
$A_3, B_3\not\in r_1\cup r_2\cup r'_1\cup r'_2$.

\end{enumerate}

\item[(2)]

Now assume that $B_i$, $i=1,2,3$, is the point on $\ell_i$ such that 
$P_{i1},P_{i2},A_i,B_i$, in this order, are harmonic. If the points of $Z$ are in LGP, then $\mathcal W(Z)$
is a quartic surface which contains at least 27 lines,
including the 25 classical ones and 
the lines $m_1$ and $m_2$.
\end{enumerate}
\end{proposition}

\begin{proof}
We represent in  Figure \ref{fig:moore1} the assumption in the statement.
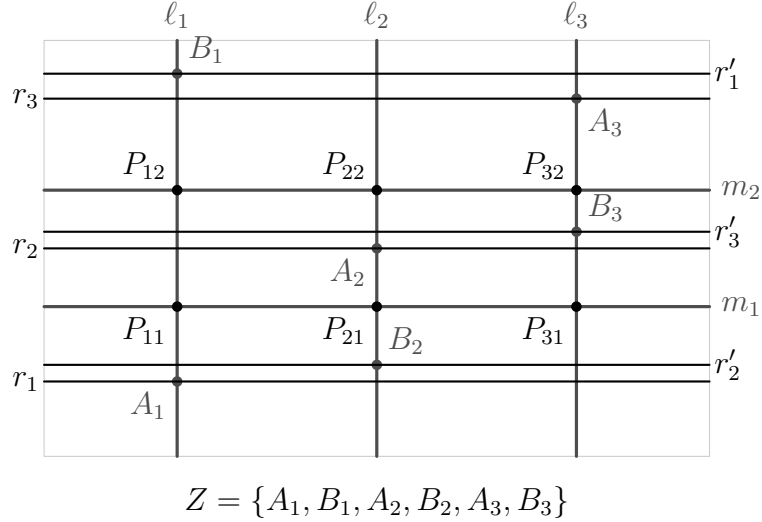
\begin{figure}
    \centering
\begin{tikzpicture}[scale=1.1, line cap=round, line join=round]

  \draw[gray!40] (0,0) rectangle (8,5);

  \def\xone{1.6}
  \def\xtwo{4.0}
  \def\xthree{6.4}

  \draw[very thick, black!70] (\xone,0)   -- (\xone,5) node[pos=1, above] {$\ell_1$};
  \draw[very thick, black!70] (\xtwo,0)   -- (\xtwo,5) node[pos=1, above] {$\ell_2$};
\draw[very thick, black!70] (\xthree,0) -- (\xthree,5) node[pos=1, above] {$\ell_3$};

  \def\yone{1.8}
  \def\ytwo{3.2}

  \draw[very thick, black!70] (0,\yone) -- (8,\yone) node[pos=1, right] {$m_1$};
  \draw[very thick, black!70] (0,\ytwo) -- (8,\ytwo) node[pos=1, right] {$m_2$};

  \foreach \i/\x in {1/\xone, 2/\xtwo, 3/\xthree}{
    \fill[black] (\x,\yone) circle (1.8pt)
      node[below left] {$P_{\i 1}$};
    \fill[black] (\x,\ytwo) circle (1.8pt)
      node[above left] {$P_{\i 2}$};
  }

  \def\yAone{0.9}
  \def\yAtwo{2.5}
  \def\yAthree{4.3}

  \coordinate (A1) at (\xone,\yAone);
  \coordinate (A2) at (\xtwo,\yAtwo);
  \coordinate (A3) at (\xthree,\yAthree);

  \def\yBone{4.6}
  \def\yBtwo{1.1}
  \def\yBthree{2.7}

  \coordinate (B1) at (\xone,\yBone);
  \coordinate (B2) at (\xtwo,\yBtwo);
  \coordinate (B3) at (\xthree,\yBthree);

  \foreach \P/\lab/\pos in {A1/$A_1$/below left,
                           A2/$A_2$/below left,
                           A3/$A_3$/below right,
                           B1/$B_1$/above right,
                           B2/$B_2$/above right,
                           B3/$B_3$/above right}
  {
    \fill[black!70] (\P) circle (1.8pt) node[\pos] {\lab};
  }

  \draw[thick, black!55!black] (0,\yAone)   -- (8,\yAone)   node[pos=0.01, left] {$r_1$};
  \draw[thick, black!55!black] (0,\yAtwo)   -- (8,\yAtwo)   node[pos=0.01, left] {$r_2$};
  \draw[thick, black!55!black] (0,\yAthree) -- (8,\yAthree) node[pos=0.01, left] {$r_3$};

  \draw[thick, black!55!black] (0,\yBone)   -- (8,\yBone)   node[pos=0.99, right] {$r'_1$};
  \draw[thick, black!55!black] (0,\yBtwo)   -- (8,\yBtwo)   node[pos=0.99, right] {$r'_2$};
  \draw[thick, black!55!black] (0,\yBthree) -- (8,\yBthree) node[pos=0.99, right] {$r'_3$};

  \node[black] at (4,-0.55) {$Z=\{A_1,B_1,A_2,B_2,A_3,B_3\}$};

\end{tikzpicture}

    \caption{Three sets of harmonic points in Proposition \ref{p.Weddle >25} on the rulings of a quadric.}
   \label{fig:moore1}
\end{figure}
(1)
First assume (a).
We argue contrapositively.
Let $1\leq i < j\leq 3$ and let $k$ be such that $\{i,j,k\}=\{1,2,3\}$.
If $A_j\in r_i$, then $\ell_k\cup r_i$ defines a plane that contains
$A_k,B_k,A_i,A_j$, contrary to assumption, while
if $A_j\in r'_i$, then $\ell_k\cup r'_i$ defines a plane that contains
$A_k,B_k,A_j,B_i$, contrary to assumption. The same argument holds for the $B_j$. 

Now assume (b).  
Assume, by contradiction, that four of the points are coplanar
(contained in a plane $\Pi$ say).

The four points must be three $A$'s and a $B$, three $B$'s and an $A$, or two of each.
If it is three $A$'s and a $B$, or three $B$'s and an $A$, then 
$\Pi$ contains an $A_i$ and a $B_i$ on the same line $\ell_i$, hence
$\Pi$ contains the line $\ell_i$. This is a line in a ruling on $\mathcal Q$,
so $\Pi\cap \mathcal Q$ is $\ell_i\cup L$ where $L$ is a line in the other ruling.
In the case of three $A$'s and a $B$, $L$ must contain the two $A$ points
not on $\ell_i$. 
These two points are $A_j, A_k$ for some $1\leq j<k\leq 3$, 
so $L=r_j=r_k$ and we have $A_k\in r_j$, contrary to assumption.
The case of three $B$'s and an $A$ is completely analogous. 

Now say $\Pi$ contains two $A$'s and two $B$'s.
Then $\Pi$ must contain $A_i$ and $B_i$ (and hence $\ell_i$) for some $i$,
so $\Pi\cap \mathcal Q=\ell_i\cup L$ for a line $L$ in the other ruling.
Since no plane contains skew lines, the remaining two points
cannot be $A_j$ and $B_j$ for any $j$, so they are
$A_j$ and $B_k$ such that
$\{i,j,k\}=\{1,2,3\}$. In particular, 
$L$ contains $A_j$ and $B_k$, so $r_j=L=r'_k$, and hence $r'_j=r_k$.

If $k<j$, then $A_j\in r'_k$, contrary to assumption.
If $j<k$, then $A_k\in r'_j$, contrary to assumption.

(2) 
We now check that $m_1,m_2\subset\mathcal W(Z)$.
Let $O$ be a general point of $m_1$.
Projecting from $O$, the images of the points $P_{11},P_{21}, P_{31}$ coincide. Since, after the projection, $P_{12},P_{22}, P_{32}$ are on a line and the harmonic property is preserved, the images of the 6 points 
$A_1,B_1,A_2,B_2,A_3,B_3$ lie on a conic by Lemma \ref{l. harmonic and conic}, 
and hence $O$ (and thus $m_1$) is contained in $\mathcal W(Z)$. 
The same argument holds for $m_2$, applied to the quadruple
$P_{i2}, P_{i1}, A_i, B_i$ on each line $\ell_i$.
Just note that the cross-ratio of $P_{i2}, P_{i1}, A_i, B_i$ is the reciprocal
of the cross ratio for $P_{i1},P_{i2},A_i,B_i$, but the latter are harmonic
so have cross ratio $-1$, hence $P_{i2}, P_{i1}, A_i, B_i$ also
have cross ratio $-1$ and so are harmonic.

Finally, we check that neither $m_1$ nor $m_2$ can be among the 25 lines
mentioned earlier. First, no point of $Z$, by construction, is on
$m_1$ or $m_2$, so neither line can be among the lines through pairs of points of $Z$.
And second, both $m_1$ and $m_2$ are ruling lines in the ruling of $\mathcal Q$ which does 
not include the $\ell_i$. So any plane containing $m_i$ meets $\mathcal Q$ in a line in the ruling
which includes the $\ell_i$, and each such line contains either 2 points of $Z$
(if the line is one of the $\ell_i$) or it contains no points of $Z$.
In particular, no plane containing $m_i$ can contain 3 points of $Z$,
so neither line $m_i$ can occur as the intersection of two planes whose union
contains $Z$.
\end{proof}

\begin{remark}\label{rem:veronese-harmonic}
In the special situation where $Z$ lies on the twisted cubic
\[
\nu_3\colon \PP^1\longrightarrow \PP^3,\qquad [s:t]\longmapsto [s^3:s^2t:st^2:t^3],
\]
the “harmonic condition” in Proposition~\ref{p.Weddle >25} can be read directly on $\PP^1$; this is a consequence of Desargues’ involution theorem (cf. \cite[Theorem 9.31]{CoxeterPG}).

Indeed, every projective automorphism $\varphi\in\PGL_2(\CC)$ induces a projective automorphism
$\Sym^3(\varphi)\in \PGL_4(\CC)$ preserving $\nu_3(\PP^1)$ and acting on it via
$\nu_3\circ \varphi=\Sym^3(\varphi)\circ \nu_3$. 

Conversely, every projective automorphism of \(\PP^3\) preserving \(\nu_3(\PP^1)\) arises in this way (uniquely) from an element of \(\PGL_2(\CC)\).

In particular, if $X=\{p_1,\dots,p_6\}\subset \PP^1$ is partitioned into three pairs
$\{p_1,p_2\}\cup\{p_3,p_4\}\cup\{p_5,p_6\}$ and there exists an involution
$\varphi\in \PGL_2(\CC)$ with
\[
\varphi(p_1)=p_2,\qquad \varphi(p_3)=p_4,\qquad \text{and}\qquad  \varphi(p_5)=p_6,
\]
then passing to $\PP^3$ via $\nu_3$, the induced involution $\Sym^3(\varphi)$ preserves the three secant
lines $\ell_i=\langle \nu_3(p_{2i-1}),\nu_3(p_{2i})\rangle$ and acts on each $\ell_i\cong \PP^1$
as the involution swapping the two marked points. This forces the existence of two lines $m_1,m_2$ in the complementary  ruling of the (unique) smooth quadric
$\mathcal Q$ containing $\ell_1,\ell_2,\ell_3$ such that
\[
(\ell_i\cap m_1,\ \ell_i\cap m_2;\ \nu_3(p_{2i-1}),\nu_3(p_{2i}))=-1\qquad (i=1,2,3).
\]
The lines $m_1$ and $m_2$ satisfying the harmonic condition  are uniquely determined (up to swapping $m_1$ and $m_2$) by the partition.

Conversely, the existence of such $m_1,m_2$ forces an involution of $\PP^1$ swapping the three pairs, hence an induced involution of $\PP^3$ preserving $\nu_3(\PP^1)$ and stabilizing $\{m_1,m_2\}$.
\end{remark}

The number of lines in a Weddle surface can also be larger than $27$. Indeed,  Edge in \cite{Edge} provides an example where such number is at least $37$, i.e., there are twelve extra lines contained in the Weddle surface beside the canonical ones. However, Edge does not write down explicit equations for these lines.  In the next example we apply Proposition~\ref{p.Weddle >25} to recover eight of these twelve extra lines predicted by Edge. 

\begin{example}[The Weddle surface with the dihedral symmetry]\label{ex:37lines}
Let $\varepsilon$ be a primitive sixth root of unity, $\varepsilon^6=1$, and let $Z$ be the set of six points on the twisted cubic
\[
[s:t]\longmapsto [s^3:s^2t:st^2:t^3]\subset\PP^3
\]
corresponding to $q_i=[1:\varepsilon^i]$ for $i=0,\ldots,5$, i.e.,
\[
P_i=[1:\varepsilon^i:\varepsilon^{2i}:\varepsilon^{3i}]\in\PP^3.
\]
Since the $P_i$ lie on a twisted cubic, they are in linear general position in $\PP^3$.

In this case the Weddle surface has equation 
$$\mathcal W(Z)\colon 2(xy^3+z^3w)-(3yz-xw)(x^2+w^2)=0,$$ 
and it contains, according to Edge \cite{Edge}, twelve further lines in addition to the $25$ canonical ones; hence $\mathcal W(Z)$ contains at least $25+12=37$ lines. We show how to find eight of them as an application of Proposition \ref{p.Weddle >25}.  

Every projective automorphism of the twisted cubic is induced by an automorphism of $\PP^1$.
Thus the subgroup of $\PGL_4(\CC)$ preserving $Z$ is the image of the stabilizer of the set
$\{[1:\varepsilon^i]\ |\ i=0,\ldots, 5\}\subset \PP^1$ under the $\Sym^3$-representation of $\PGL_2(\CC)$.
This stabilizer is the dihedral group with twelve elements, $D_{12}$. 
Concretely, it is generated by the rotation $[s:t]\mapsto [s:\varepsilon t]$ and the involution
$[s:t]\mapsto [t:s]$, whose images in $\PGL_4(\CC)$ are represented by
\[
\mathcal R=\begin{pmatrix}
1&0&0&0\\
0&\varepsilon&0&0\\
0&0&\varepsilon^2&0\\
0&0&0&\varepsilon^3
\end{pmatrix}
\qquad\text{and}\qquad
\mathcal S=\begin{pmatrix}
0&0&0&1\\
0&0&1&0\\
0&1&0&0\\
1&0&0&0
\end{pmatrix}.
\]
There are $15$ partitions of $Z$ into three unordered pairs.
We group them into $D_{12}$-orbits and for each orbit we pick one representative
\(
M=\{\{i_1,j_1\},\{i_2,j_2\},\{i_3,j_3\}\}.
\)

When the hypotheses of Proposition~\ref{p.Weddle >25} hold for the corresponding three secant lines
\(\ell_k=\langle P_{i_k},P_{j_k}\rangle\),
we obtain two lines \(m_1(M),m_2(M)\subset \mathcal W(Z)\).
Since $D_{12}$ preserves $Z$, it preserves $\mathcal W(Z)$, hence the full orbit
\(\mathrm{Orb}_{D_{12}}(M)\)
produces the lines \(g(m_1(M))\) and \(g(m_2(M))\) for \(g\in D_{12}\).

\begin{itemize}
    \item {\bf Orbit size 1.} Consider the partition
\[
M_0=\{ \{0,3\},\{1,4\},\{2,5\}\},
\]
i.e.,
\[
Z=\{P_0,P_3\}\ \cup\ \{P_1,P_4\}\ \cup\ \{P_2,P_5\}.
\]
On $\PP^1$ this corresponds to the partition of the six points
$q_i=[1:\varepsilon^i]$ into the three pairs $\{q_i,q_{i+3}\}$.
It is preserved by the involution
\[
\varphi_0:\ [s:t]\longmapsto [s:-t],
\]
since $\varphi_0(q_i)=q_{i+3}$ for all $i$.

By Remark~\ref{rem:veronese-harmonic} the harmonic condition in Proposition~\ref{p.Weddle >25} holds for $M_0$,
hence we obtain two lines $m_1(M_0),m_2(M_0)\subset \mathcal W(Z)$.

More precisely, the corresponding secant lines $\ell_i=\langle P_i, P_{i+3}\rangle$ for $i=0,1,2$, contained in the smooth quadric
$\mathcal Q:\ xw-yz=0,$ are
\[
\ell_0:\begin{cases}
     z-x=0\\ w-y=0,
\end{cases}
\qquad
\ell_1:\begin{cases}
    z-\varepsilon^{2}x=0\\ w-\varepsilon^{2}y=0,
\end{cases} 
\qquad
\ell_2:\begin{cases}
    z-\varepsilon^{4}x=0\\ w-\varepsilon^{4}y=0.
\end{cases} 
\]
And \[
m_1(M_0):\ y=w=0,
\qquad
m_2(M_0):\ x=z=0.
\]

\item \textbf{Orbit size 2.}
A representative is
\[
M_1=\{\{0,5\},\{3,4\},\{1,2\}\},
\]
whose orbit consists of
\[
\{\{0,5\},\{3,4\},\{1,2\}\},\qquad
\{\{0,1\},\{2,3\},\{4,5\}\}.
\]

For $M_1$ the involution swapping points in the first two pairs is represented by the matrix
\[
\varphi_1=
\begin{pmatrix}
-1 & -1+2\varepsilon\\
-1+2\varepsilon & 1
\end{pmatrix}\in\PGL_2(\CC).
\]
By construction we have
\[
\varphi_1(q_0)=q_5,\qquad \varphi_1(q_5)=q_0,\qquad
\varphi_1(q_3)=q_4,\qquad \varphi_1(q_4)=q_2,
\]
but, $\varphi_1$ does not preserve the third pair:
\[
\varphi_1(q_1)\neq q_2.
\]
Hence this orbit does not
produce additional lines via Proposition~\ref{p.Weddle >25}.

\item \textbf{Orbit size 3 (I).} A representative is
\[
M_2=\{\{0,1\},\{2,5\},\{3,4\}\},
\]
with orbit
\[
\{\{0,1\},\{2,5\},\{3,4\}\},\quad
\{\{0,3\},\{1,2\},\{4,5\}\},\quad
\{\{0,5\},\{1,4\},\{2,3\}\}.
\] On $\PP^1$ this partition is preserved by the reflection
\[
\varphi_2\colon \ [s:t]\longmapsto [t:\varepsilon s],
\]
since $\varphi_2(q_i)=q_{1-i}$.
By Remark~\ref{rem:veronese-harmonic}, Proposition~\ref{p.Weddle >25} applies and yields
two lines $m_1(M_2),m_2(M_2)\subset\mathcal W(Z)$.
For the chosen representative we get
\[
m_1(M_2):\ 
\begin{cases}
w+i\,x=0\\
\varepsilon\,y+z=0,
\end{cases}
\qquad
m_2(M_2):\ 
\begin{cases}
w-i\,x=0\\
y+\varepsilon\,z=0.
\end{cases}
\]
Since the orbit of $M_2$ has size $3$, applying $\mathcal R$ produces the  lines
\[
\mathcal R(m_1(M_2)):\ 
\begin{cases}
w-i\,x=0\\
\varepsilon^2\,y+z=0,
\end{cases}
\qquad\ \mathcal R(m_2(M_2)):\ 
\begin{cases}
w+i\,x=0\\
\varepsilon^2\,y+z=0,
\end{cases}
\]\[ \mathcal R^2(m_1(M_2)):\ 
\begin{cases}
w+i\,x=0\\
-y+z=0,
\end{cases}
\qquad\ \mathcal R^2(m_2(M_2)):\ 
\begin{cases}
w-i\,x=0\\
-y+z=0.
\end{cases}
\]

\item \textbf{Orbit size 3 (II).}
A representative is
\[
M_3=\{\{0,2\},\{1,4\},\{3,5\}\},
\]
with orbit
\[
\{\{0,2\},\{1,4\},\{3,5\}\},\quad
\{\{0,3\},\{1,5\},\{2,4\}\},\quad
\{\{0,4\},\{1,3\},\{2,5\}\}.
\]
The unique involution that swaps the first two pairs is represented by the matrix  
\[
\varphi_3 \;=\;
\begin{pmatrix}
2 & \varepsilon-1\\
\varepsilon & -2
\end{pmatrix}\in \PGL_2(\CC).
\]
By construction it satisfies
\[
\varphi_3(q_0)=q_2,\quad \varphi_3(q_2)=q_0,\quad
\varphi_3(q_1)=q_4,\quad \varphi_3(q_4)=q_1.
\]
But $\varphi_3(q_3)\neq q_5$.

\item \textbf{Orbit size 6.}
A representative is
\[
M_4=\{\{0,1\},\{2,4\},\{3,5\}\},
\]
with orbit
\[
\begin{aligned}
&\\
&\Bigl\{
\{\{0,1\},\{2,4\},\{3,5\}\},\;
\{\{0,2\},\{1,3\},\{4,5\}\},\;
\{\{0,2\},\{1,5\},\{3,4\}\},\\
&\{\{0,4\},\{1,2\},\{3,5\}\},\;
\{\{0,4\},\{1,5\},\{2,3\}\},\;
\{\{0,5\},\{1,3\},\{2,4\}\}
\Bigr\}.
\end{aligned}
\]
For $M_4$ the involution swapping points in the first two pairs is represented by
\[
\varphi_4=
\begin{pmatrix}
-2 & -1+3(\varepsilon-\varepsilon^5)\\[2pt]
1+3(\varepsilon-\varepsilon^5) & 2
\end{pmatrix}\in\PGL_2(\mathbb C).
\]
However,
$\varphi(q_3)\notin \{q_0,\dots,q_5\},$ so this partition does not contribute additional lines via Proposition~\ref{p.Weddle >25}.
\end{itemize}
\end{example}

There is another set of six points whose stabilizer gives a larger group, and a set of twelve extra lines in its Weddle surface, we describe it in the next example. 

\begin{example}[The Weddle surface and the octahedral symmetry]\label{6points Octahedral}
Let $i$ be a primitive fourth root of unity, $i^2=-1$, and let $Z$ be the set of six points on the twisted cubic
\[
[s:t]\longmapsto [s^3:s^2t:st^2:t^3]\subset\PP^3
\]
corresponding to $[1:0]$, $[0:1]$ and  $[1:1]$, $[1:i]$, $[1:-1]$, $[1:-i]$, i.e.,
\[
\begin{matrix}
    P_0=[1:0:0:0], P_\infty=[0:0:0:1],\\ P_1=[1:1:1:1], P_2=[1:i:-1:-i],P_3=[1:-1:1:-1], P_4=[1:-i:-1:i]  \in\PP^3.
\end{matrix}
\]
One can compute that the Weddle surface of $Z$ has equation
\[
\mathcal W(Z):\ (x^{2}-z^{2})(wy-z^2)+(w^{2}-y^2)(xz-y^{2})=0.
\]

Since the $P_*$ lie on a twisted cubic, they are in linear general position in $\PP^3$. Every projective automorphism of the twisted cubic is induced by an automorphism of $\PP^1$.
Thus the subgroup of $\PGL_4(\CC)$ preserving $Z$ is the image of the stabilizer of
\[
X=\{0,\infty,1,-1,i,-i\}\subset\PP^1
\]
under the $\Sym^3$-representation of $\PGL_2(\CC)$.
This stabilizer is isomorphic to the octahedral group $S_4$ (of order $24$). A convenient generating pair is given by the maps
\[
[s:t]\mapsto [s:i\,t]\quad\text{and}\quad
[s:t]\mapsto [s-t:s+t],
\]
which preserve the set $\{[1:0],[0:1],[1:\pm1],[1:\pm i]\}$.
Under $\Sym^3$, these induce the projective maps of $\PP^3$ represented by the matrices
\[
\mathcal R=
\begin{pmatrix}
1&0&0&0\\
0&i&0&0\\
0&0&-1&0\\
0&0&0&-i
\end{pmatrix},
\qquad
\mathcal S=
\begin{pmatrix}
1&-1&1&-1\\
1&-1&-1&1\\
1&1&-1&-1\\
1&1&1&1
\end{pmatrix}.
\]
In particular, $\langle \mathcal R,\mathcal S\rangle \cong S_4$ and it preserves the set $Z$.

As in Example~\ref{ex:37lines}, there are $15$ partitions of $Z$ into three unordered pairs. 
Up to the $S_4$-action, the $15$ partitions split into three orbit types:
\[
\text{orbit sizes }1,\ 6,\ 8.
\]

\begin{itemize}
\item \textbf{Orbit size 1.}
A representative is the ``antipodal'' partition
\[
M_0=\{\{0,\infty\},\{1,-1\},\{i,-i\}\}.
\]
There is no involution of $\PP^1$ swapping each of these three pairs simultaneously:
indeed, any involution swapping $0\leftrightarrow\infty$ is of the form $$\varphi_0(z)=\begin{pmatrix}
    0 & 1\\
    k & 0\\
\end{pmatrix}\in \PGL_2(\mathbb C)$$
and imposing $\varphi_0([1:1])=[1:-1]$ forces $k=-1$, which fixes $[1:i]$ rather than swapping $[1:i]\leftrightarrow [1:-i]$.
Hence Proposition~\ref{p.Weddle >25} does not apply to $M_0$.

\item \textbf{Orbit size 6.}
A representative is
\[
M_1=\{\{0,\infty\},\{1,i\},\{-1,-i\}\}.
\]
In this case we get an involution, that is represented by
\[
\varphi_1= \begin{pmatrix} 0 & i\\ 1 & 0\end{pmatrix}\in \PGL_2(\mathbb C)
\]
and satisfies 
\[
\varphi_1([0:1])=[1:0],\ \qquad
\varphi([1:1])=[0:i],\ \qquad
\varphi([1:-1])=[1:-i].
\]
Therefore the harmonic hypothesis in Proposition~\ref{p.Weddle >25} holds for this partition,
so the corresponding Weddle surface $\mathcal W(Z)$ contains two additional lines
$m_1(M_1)$ and $m_2(M_1)$ arising from that proposition, precisely: 
\[m_1(M_1):\ 
\begin{cases}
\eta\,y-z=0\\
x+\eta\,w=0,
\end{cases}
\qquad
m_2(M_1):\ 
\begin{cases}
\eta\,y+\,z=0\\
x-\eta\,\,w=0.
\end{cases}\]
where $\eta$ is an 8-th primitive root of unity.

Since the group $S_4$ preserves $Z$ (hence $\mathcal W(Z)$), applying the elements of $S_4$
to $m_1(M_1)$ and $m_2(M_1)$ produces a collection of extra lines; in particular,
this orbit contributes to $2\cdot 6=12$ lines.

\item \textbf{Orbit size 8.}
A representative is, for instance,
\[
M_2=\{\{0,1\},\{\infty,i\},\{-1,-i\}\}.
\]
One checks that there is no involution $\varphi\in\PGL_2(\CC)$ swapping simultaneously
$[1:0]\leftrightarrow [1:1]$ and $[0:1]\leftrightarrow [1:i]$ and $[1:-1]\leftrightarrow [1:-i]$.
Hence Proposition~\ref{p.Weddle >25} does not apply to this orbit.
\end{itemize}
\end{example}

It has been proved by B. Segre that there are at most $64$ lines on a smooth complex quartic surface, and this bound is sharp (it is attained by Schur's quartic). This bound has been only recently extended to nodal quartic surfaces in $\mathbb P^3$ by Veniani \cite[Theorem 1]{Veniani2017}. An optimal bound for $K3$ quartic surfaces with at least one singular point is not
known. The largest \emph{explicit} number currently realized is $39$ lines achieved by a Delsarte surface (see Example~7.2 in \cite{Veniani2017}). 
It is reasonable to expect that the maximal number of lines on Weddle surfaces (like in the case of all quartic surfaces) is obtained on the most symmetrical ones. This prompts us to believe that the number of lines we got in Example \ref{6points Octahedral} is maximal. 
\begin{conjecture}
    Let $Z$ be a set of six points in LGP in $\mathbb P^3$. Then $\mathcal W(Z)$ contains at most 37 lines.
\end{conjecture}

\begin{remark}\label{Remark4.9}
    The previous examples show that the number of lines on the Weddle surface can vary  
    for six LGP points. However, the singular locus of the Weddle surface of six LGP points always consists exactly
    of the six points. Indeed, the following proposition shows that they are ordinary double points. The fact that the surface is otherwise smooth follows by a calculation (see Appendix \ref{AppendixA}).
    \end{remark}

\begin{proposition} \label{prop: ordnodes} Let $Z=\{P_1,\dots,P_6\}$ be
a set of $6$ points in LGP in $\mathbb P^3$.
Then the Weddle surface $\mathcal W(Z)$ has ordinary nodes as
singularities at each point $P_i$.
\end{proposition}
\begin{proof} After renumbering, we can reduce ourselves to the case $i=1$.

First, we claim that the multiplicity of $P_1\in\mathcal W(Z)$ is $2$.
Indeed, let $\Pi$ be the plane spanned by
$P_1,P_2,P_3$. Then $\Pi$ meets $\mathcal W(Z)$ in a quartic curve $Q$
which contains the lines
$\langle P_1,P_2\rangle$, $\langle P_1,P_3\rangle$, $\langle
P_2,P_3\rangle$. Thus $Q$ splits in the three previous lines plus a
fourth line $L$.
Observe that $\mathcal W(Z)$ contains the line $L'$ of intersection of
$\Pi$ with the plane $\Pi'$ spanned by $P_4,P_5,P_6$,
since a projection of $Z$ from a general point of $L'$ lies in a
reducible conic.
Since $Z$ is in LGP, then $\Pi'$ cannot contain any of the points $P_1,
P_2,P_3$. Then necessarily $L=L'$.
Hence $L$ misses $P_1$, and $P_1$ has multiplicity $2$ in $\mathcal
W(Z)\cap \Pi$. This proves the claim.

Next, we claim that the (quadratic) tangent cone of $\mathcal W(Z)$ at
$P_1$ is irreducible.
Indeed, the tangent cone contains the $5$ lines $\langle
P_1,P_i\rangle$, $i=2,\dots,5$, which lie in $\mathcal W(Z)$, and
observe that
no three of these lines are coplanar, since $Z$ is in LGP. This
completes the proof of the statement.
\end{proof}
So far, we were interested in Weddle surfaces maximizing the number of lines they contain. The next result goes in the opposite direction. It excludes that the Weddle surface $\mathcal W(Z)$ of $Z$ contains any
line through only one of the points in a set $Z$
of six points.
	
\begin{lemma}\label{l.line in Weddle}
	Let $Z=\{P_1,\ldots,P_6\}\subseteq \PP^3$ be a set of mutually distinct points in LGP.
	Let $L$ be a line in $\mathcal W(Z)$. Then the number of points in the intersection
    $$L\cap Z$$
    is either $0$ or $2$.
\end{lemma}	
\begin{proof}
Assume to the contrary that $L$ is a line contained in $\mathcal W(Z)$ such that (up to renumbering the points again)
$$P_1\in L \;\mbox{ and }\; 
P_2,\ldots,P_6\notin L.$$
Consider the planes 
\[\alpha=\langle P_2,P_3,P_4\rangle\ \ \text{and} \ \ \beta=\langle P_2,P_5,P_6\rangle\] 
and let \[P=L\cap \beta.\]
The line $L$ is not contained in $\beta$ because the points $P_1,P_2,P_5,P_6$ are not coplanar.
Moreover, we claim that $P$ is not contained in any of the lines $\langle P_2,P_5 \rangle$, $\langle P_2, P_6\rangle$ and $\langle P_5,  P_6\rangle$.  To this end assume (up to renumbering the points) $P\in\langle P_2,P_5\rangle$.
Then $L$ is contained in the plane $F=\langle P_1,P_2,P_5\rangle$. Thus the intersection $F\cap \mathcal{W}(Z)$ splits into four lines:
$$\langle P_1, P_2\rangle
\cup \langle P_1,P_5\rangle
\cup \langle P_2,P_5\rangle
\cup L.$$
Since $\mathcal{W}(Z)$ has an ordinary node at each of the points $P_1,P_2,P_5$, its tangent cone at each of these points is a nondegenerate quadric cone. Hence it cannot contain three coplanar lines.

As a consequence, the projection $\pi_P$ from the point $P$ to the plane $\alpha$
maps $P_2,P_5,P_6$ to three distinct collinear points (because $P$ lies in the plane $\beta$).
The images of the points $P_2,P_5,P_6$ are distinct because we have just argued that $P$ is not on any
line through two of the points. 

But $P$ is in the Weddle surface of $Z$, so, the conic through  $\pi_P(Z)$ splits into the union of two lines, one containing $P_2=\pi_P(P_2), \pi_P(P_5), \pi_P(P_6)$ and the other through the points 
$$\pi_P(P_1)=L\cap \alpha,\; P_3=\pi_P(P_3)\;\mbox{ and }\; P_4=\pi_P(P_4).$$
But then $L$ (and hence $P_1$)
and also $P_3, P_4$ are coplanar. However by the above token it is not possible that three coplanar lines in $\mathcal{W}(Z)$ (here $L$, $\langle P_1,P_3 \rangle$ , $\langle P_1, P_4\rangle$) pass through $P_1$. A contradiction.
\end{proof}
As an application of Lemma \ref{l.line in Weddle} we get the following result which will be useful later.

\begin{proposition} \label{seven pts}
	Let $X=\{P_1,\ldots,P_7\}\subseteq \PP^3$ be a set of seven points in LGP.
	Then the 2-Weddle locus of $X$  is a curve and it does not contain as a component any line joining two points of $X$.
\end{proposition}
\begin{proof}
Let $\mathcal C$ be the set of cones in $[I(X)]_2$ and let $\mathfrak c$ be the set of the vertices of the cones in $\mathcal C$, i.e., $\mathfrak c $ is the 2-Weddle locus of $X$.  	
From \cite{maroscia} the $h$-vector of $X$ is $(1,3,3)$.
Thus $\PP([I(X)]_2)\cong \PP^2$. Since the family of singular quadrics in $\PP^3$ is a hypersurface and contains no planes,  $\mathcal C$ has codimension 1 in $[I(X)]_2$. So $\mathfrak c$ is a curve in $\PP^3$. 

We need to exclude that  $\mathfrak c$ contains any line through two points of $X.$
Assume, by contradiction, that the line $L=\langle P_1, P_2\rangle$ is a component of $\mathfrak c$.
So, after projecting from a point in $L,$ the point $P_1$ collides with $P_2$ and $P_2,\ldots,P_7$ are on a conic.
Hence, the line $L$ is contained in the Weddle surface of $P_2,\ldots,P_7$ and this contradicts Lemma~\ref{l.line in Weddle}.
\end{proof}

We will see in Section \ref{sec:+1} that the curve obtained in Proposition \ref{seven pts} is ACM of degree 6 and genus 3. 


\subsection{A totally reducible Weddle surface} \label{totally reducible}

In Section \ref{more25 lines} we saw that there exist sets of six points for which the Weddle surface does not behave in the expected way, even though LGP holds. From now on, we consider sets of six points not in LGP and we record an even stranger behavior for their Weddle surfaces, while still having degree 4. 

Before turning to details we recall that for integers $a,b\geq 1$ an $(a,b)$-grid in $\PP^3$ is the set of $ab$ intersection points of a family $a$ skew lines with another family of $b$ skew lines.

\begin{proposition}\label{p. reducible Weddle surface and harmonic}
Let $L_1,L_2,L_3$ be three noncoplanar lines concurrent at a point $O$.
Let $Z=\{P_1,P_2,P_3, Q_1,Q_2,Q_3\}$ be a set of six points in $\PP^3\setminus\{O\}$,
with $P_i,Q_i\in L_i$ for $i=1,2,3$.
For each $i$ let $H_i\in L_i$ be the harmonic conjugate of $O$ with respect to $P_i,Q_i$.
\noindent
Then the Weddle surface of $Z$ is the union of the four planes
\[
\mathcal W(Z)=\langle L_1,L_2\rangle \ \cup\ \langle L_1,L_3\rangle \ \cup\ \langle L_2,L_3\rangle \ \cup\ \langle H_1,H_2,H_3\rangle.
\] 
\end{proposition}

\begin{proof}We first note that the general  projection of $Z$  is not contained in a conic since $Z$ is not a (2,3)-grid (cf. \cite{CM} Proposition 4.3). 

Now we prove that the four planes in the statement are contained in $\mathcal W(Z).$
For a general point $P$ on $\langle L_i,L_j\rangle$ the projection $\pi_P$ maps the four points on $L_i\cup L_j$ on four points on a line; thus $Z$ on a reducible conic. 
Furthermore, for a general point $P$ on the plane spanned by $H_1,H_2,H_3$, the projection $\pi_P$ maps the points $H_1,H_2,H_3$ into a line and thus, by Lemma \ref{l. harmonic and conic}, $Z$ into a conic.

Finally, $Z$ has $h$-vector $(1,3,2)$ hence the $2$-Weddle scheme is a surface of degree $4$.
Since we have exhibited four distinct planes contained in it, it cannot have any additional component without increasing the degree. Therefore the $2$-Weddle scheme is exactly the above union of planes, and the same holds for the $2$-Weddle locus.
\end{proof}
In the following example we exhibit the interpolation matrix for sets as in Proposition~\ref{p. reducible Weddle surface and harmonic}.
\begin{example}\label{e. reducible Weddle surface}
 Let $Q=[x:y:z:w]\in \PP^3$ be the (variable) point from which we project.
We write $\Lambda(Z+2Q,2)$ with respect to the monomial basis
\[
x^2,\; y^2,\; z^2,\; w^2,\; xy,\; xz,\; xw,\; yz,\; yw,\; zw
\]
of $[R]_2.$
 Take
   $$O=[0:0:0:1],\; P_1=[1:0:0:0],\; P_2=[0:1:0:0],\; P_3=[0:0:1:0]$$
   and
   $$Q_1=[a:0:0:1],\; Q_2=[0:b:0:1],\; Q_3=[0:0:c:1]$$
   for some nonzero numbers $a,b,c$.
   Then the interpolation matrix defining the Weddle surface has the form
{\small$$\begin{pmatrix}
1 & 0 & 0 & 0 & 0 & 0 & 0 & 0 & 0 & 0\\
0 & 1 & 0 & 0 & 0 & 0 & 0 & 0 & 0 & 0\\
0 & 0 & 1 & 0 & 0 & 0 & 0 & 0 & 0 & 0\\
a^2 & 0 & 0 & 1 & 0 & 0 & a & 0 & 0 & 0\\
0 & b^2 & 0 & 1 & 0 & 0 & 0 & 0 & b & 0\\
0 & 0 & c^2 & 1 & 0 & 0 & 0 & 0 & 0 & c\\
2x & 0 & 0 & 0 & y & z & w & 0 & 0 & 0\\
0 & 2y & 0 & 0 & x & 0 & 0 & z & w & 0\\
0 & 0 & 2z & 0 & 0 & x & 0 & y & 0 & w\\
0 & 0 & 0 & 2w & 0 & 0 & x & 0 & y & z
\end{pmatrix}.$$ }  
Computing its determinant (up to a nonzero scalar) we get
\[
2xyz(bcx+acy+abz-2abcw), 
\]
and we note that the plane
$$bcx+acy+abz-2abcw=0$$
intersects the lines $L_1, L_2, L_3$ in points
$$H_1=[2a:0:0:1],\;
H_2=[0:2b:0:1],\;
H_3=[0:0:2c:1]$$
such that $(P_i,Q_i,O,H_i)$ are harmonic as proved in Proposition \ref{p. reducible Weddle surface and harmonic}.
\end{example}


\subsection{A nonreduced  equidimensional Weddle scheme} \label{non-reduced}

Let $Z=Z_1\cup\{P\}\subset \PP^3$, where $Z_1$ consists of five general points
contained in a plane $H\cong\PP^2$, and $P\in \PP^3\setminus H$ is a general point.
Let $C\subset H$ be the (unique) conic through $Z_1$, and let $\mathcal Q\subset\PP^3$
be the quadric cone with vertex $P$ over $C$.

If $Q\in \mathcal Q\setminus Z$, then the projection from $Q$ sends the five points of $Z_1$
to five points lying on the image of $C$, hence on a conic in the target plane, and clearly it also sends $P$ to this conic; therefore
$Q\in \mathcal W(Z)$.
If instead $Q\in H$, then the projection from $Q$ sends $Z_1$ to a set of five collinear points (since $Z_1\subset H$), so again the image of $Z$ lies on a conic
(e.g., the union of that line with any line through the remaining point); hence $H\subset \mathcal W(Z)$.
On the other hand, for $Q\notin H\cup \mathcal Q$ the projected six points are not contained in any conic, so the $2$-Weddle locus is supported on $H\cup \mathcal Q$.

Since the $h$-vector of $Z_1$ is $(1,2,2)$, adjoining a general point $P$ gives the
$h$-vector of $Z$ equal to $(1,3,2)$; in particular, the $2$-Weddle scheme is cut out by
the determinant of a $4\times 4$ matrix of linear forms, hence is a quartic surface.
Because its support is contained in $H\cup \mathcal Q$ and both $H$ and $\mathcal Q$
occur as components of the locus, the quartic must contain $H$ with multiplicity at least $2$.
In particular, the $2$-Weddle scheme is equidimensional of codimension $1$ but not reduced.


\subsection{Nonmaximal minors} \label{not max minors}

In this subsection we describe another phenomenon arising from 6 points, namely a situation where the expected codimension of the ideal of maximal minors is not obtained, and in fact where we need to look at 
nonmaximal minors.

Let $L_1,L_2$ be disjoint lines in $\PP^3$ and let $P_1,P_2,P_3$ be points on $L_1$ and $Q_1,Q_2,Q_3$ points on $L_2$. Let $Z = \{P_1,P_2,P_3,Q_1,Q_2,Q_3\}$. The $h$-vector of $Z$ is $(1,3,2)$, so as  in the sequence \eqref{OrigMacDualSeqPrime}
we obtain a $4 \times 4$ Macaulay  duality matrix of linear forms. 

One would therefore expect the 2-Weddle locus to be a surface of degree 4. But notice that $Z$ is a $(2,3)$-grid, so its general projection is a complete intersection of type $(2,3)$ in $\PP^2$. Thus, the projection of $Z$ from a general point lies on a conic (in fact $\dim[I(Z_P)]_2=1$). Hence by definition the 2-Weddle locus is the set of points from which the projection of $Z$ lies on a {\it pencil} of conics. 

 Such points must lie on either $L_1$ or $L_2$ (thus obtaining an image under projection consisting of three points on a line plus one more point), so the 2-Weddle locus is $L_1 \cup L_2$. Indeed, $\dim[Z_P]_2\ge 2$ when $Z_P$ is either a set of 4 points or all its points are collinear except (at most) one.

What about the 2-Weddle scheme? The matrix obtained by \eqref{OrigMacDualSeqPrime} is  a $4 \times 4$ matrix of linear forms, but the geometry just described means that the determinant of this matrix must be zero. We thus consider the  $3 \times 3$ minors. The degree formula in Lemma \ref{J1 lemma} does not apply since the scheme defined by these nonmaximal minors does not have the expected codimension. 

We illustrate what is happening here by means of the following example.

\begin{example}\label{e.(2,3)-grid}
Consider the $(2,3)$-grid in the quadric $xw-yz$, 
{
\[
Z = \{[1:0:0:0], [0:1:0:0], [1:1:0:0], [0:0:1:0], [0:0:0:1],  [0:0:1:1] \}.
\]
}
\normalsize
The Macaulay duality matrix $B$   is
\[
B=\left [
\begin{array}{cccccc}
      z & 0 & x & 0 \\ 
      w & 0 & 0 & x \\
      0 & z & y & 0 \\
      0 & w & 0 & y 
\end{array}
\right ].
\]    
We get $\det(B)=0 $ as expected.

The $3 \times 3$ minors of $B$ are:
\[
(xzw,xw^2,-yzw,-yw^2,-xz^2,-xzw,yz^2,yzw,-xyz,-xyw,
\]\[
y^2z,y^2w,x^2z,x^2w,-xyz,-xyw).
\]  
They generate an ideal whose  primary decomposition is
\[
(x,y) \cap (z,w) \cap (w^2,z^2,y^2,x^2,yzw,xzw,xyw,xyz).
\]

Note that $(w^2,z^2,y^2,x^2,yzw,xzw,xyw,xyz) $ 
is primary for the irrelevant ideal.
Thus the saturation of the ideal of the $3\times3$ minors is $(yw, xw, yz, xz)$, 
which has primary decomposition
$(x, y) \cap  (z, w)$, 
so the 2-Weddle scheme consists of the two lines, $x=y=0$ and $w=z=0$,
which are grid lines for the (2,3)-grid $Z$.
Hence the 2-Weddle scheme is the same as the 2-Weddle locus.
\end{example}

\begin{remark} \label{compare matrices}
In the last example, 
the interpolation matrix is
\[
N = 
\left [
\begin{array}{cccccccccccccc}
1 & 0 & 0 & 0 & 0 & 0 & 0 & 0 & 0 & 0 \\ 
      0 & 0 & 0 & 0 & 1 & 0 & 0 & 0 & 0 & 0 \\
      0 & 0 & 0 & 0 & 0 & 0 & 0 & 1 & 0 & 0 \\
      0 & 0 & 0 & 0 & 0 & 0 & 0 & 0 & 0 & 1 \\
      1 & 1 & 0 & 0 & 1 & 0 & 0 & 0 & 0 & 0 \\
      0 & 0 & 0 & 0 & 0 & 0 & 0 & 1 & 1 & 1 \\
      2x & y & z & w & 0 & 0 & 0 & 0 & 0 & 0 \\
      0 & x & 0 & 0 & 2y & z & w & 0 & 0 & 0 \\
      0 & 0 & x & 0 & 0 & y & 0 & 2z & w & 0 \\
      0 & 0 & 0 & x & 0 & 0 & y & 0 & z & 2w
\end{array}
\right ].
\]
After row and column reduction (as in Remark \ref{r. ReducedIntMat})
this becomes
\[
N'=\left [
\begin{array}{cccccc}
      z & w & 0 & 0 \\ 
      0 & 0 & z & w \\
      x & 0 & y & 0 \\
      0 & x & 0 & y 
\end{array}
\right ]
\]
which is the transpose of the Macaulay duality matrix above.
We have $\det N = 0$, which we expect since the general projection of $Z$ is a complete intersection of type $(2,3)$,
so we need to look at subminors.
The ideal of $9 \times 9$ minors of $N$ gives the same ideal 
as in Example \ref{e.(2,3)-grid}.
\end{remark}


\section{\texorpdfstring{$d$}{d}-Weddle loci for some general sets of   points in \texorpdfstring{$\PP^3$}{P3}}\label{sec:WeddleGeneralPoints}

In this section we study $d$-Weddle loci in $\PP^3$ for several families of general point sets.
The guiding question is: if $Z$ is a set of general points in $\PP^3$, then for which centers of projection $P\in\PP^3$ does the projected set
$\pi_P(Z)\subset\PP^2$ fail to impose the expected number of conditions on plane curves of degree $d$?
Equivalently, for which $P$ does $\dim [I(\pi_P(Z))]_d$ jump above its generic value?
We begin by reviewing a few low-cardinality cases, then discuss and correct some claims of Emch,
and finally treat the families $|Z|=\binom{d+2}{2}+k$ for small integers $k$.

 For five general points the only way a projection can lie on more than one conic is for one of the points to ``disappear", i.e., for $P$ to lie on a line joining two of the points. Hence the 2-Weddle locus is a curve, union of the secant lines, of degree $\binom{5}{2} = 10$. 
 For six points the 2-Weddle locus is the classical Weddle surface, which has degree~4. 
 
 For seven general points, Emch claims that it is classically known that the 2-Weddle locus is a curve of degree 6 and genus 3 (we will confirm this). Emch omits the case of 8 and 9 points. For 10 points he shows (\cite[page 273]{EMCH}) that the 3-Weddle locus is a surface of degree 10. For 11 points he claims 
the 3-Weddle locus is a curve of degree 52 (\cite[page 275]{EMCH}), but this is incorrect: our argument gives a curve of degree 45, and we have confirmed this using symbolic computation software (\cite{DGPS}). 

\subsection{$d$-Weddle varieties for $\binom{d+2}{2}$ general points} 
Theorem \ref{WeddleSchemeSigmaThm},   in the case $n=2$, determines the $d$-Weddle scheme for a general set of $\binom{d+2}{2}$ points in $\PP^3$. In particular, in $\PP^3$ it implies that the $d$-Weddle scheme for $\binom{d+2}{2}$ general points
is a surface of degree $\binom{d+2}{3}$.

We will need one further refinement.  If $P$ lies on this $d$-Weddle surface,
then $\pi_P(Z)$ lies on at least one plane curve of degree $d$.
Can it happen that $\pi_P(Z)$ lies on a \emph{pencil} of degree $d$ curves,
rather than on a unique one? Equivalently, can it happen that
\[
\dim_\field [I(\pi_P(Z))]_d \ge 2,
\]
i.e., that in \eqref{OrigMacDualSeqPrime} the multiplication map
$\times \partial_{L_P}$ has cokernel of dimension at least $2$?
 We expect that the answer is  {\it no}. The next theorem shows that this does not occur for a general point of any component.

\begin{theorem}\label{t. not a pencil}
Let $Z$ be a general set of $\binom{d+2}{2}$ points in $\PP^{3}$. For a point $P \in \PP^{3}$ 
let $\pi_P$ be the projection to a general plane $H$. 
Let $\Sigma$ be the $d$-Weddle scheme of $Z$.  
For a general point $Q$ in any component of  $\Sigma$, the projection  $\pi_Q(Z)$ does not lie on a pencil of curves of degree $d$ in $H$.
\end{theorem}

\begin{proof}
Recall that $$q=\deg(\Sigma)=\binom{d+2}{3}$$ by Theorem~\ref{WeddleSchemeSigmaThm}.
For $P\in\Sigma$ we have $\rank (\times\partial_{L_P})\le q-1$, and the cokernel of $\times\partial_{L_P}$ is canonically
isomorphic to $[I(Z)\cap I(P)^d]_d$, hence also to $[I(\pi_P(Z))]_d$.
Thus
\[
\dim_{\mathbb C} [I(\pi_P(Z))]_d = \dim_{\mathbb C}\coker(\times\partial_{L_P}) = q-\rank (\times\partial_{L_P}).
\]
In particular, $\pi_P(Z)$ lies on a \emph{pencil} of degree $d$ curves if and only if
\[
\dim_{\mathbb C} [I(\pi_P(Z))]_d \ge 2
\quad\Longleftrightarrow\quad
\rank (\times\partial_{L_P})\le q-2.
\]
Equivalently, all $(q-1)\times(q-1)$ minors of a matrix representing $\times\partial_{L_P}$ vanish.
Let $\Gamma\subset\PP^3$ be the closed subscheme defined by these minors; then $\Gamma\subseteq \Sigma$.

Assume for contradiction that there exists an irreducible component $\Sigma_0$ of $\Sigma$ such that,
for a general point $P\in\Sigma_0$, the set $\pi_P(Z)$ lies on a pencil of degree $d$ curves, i.e.,
$P\in\Gamma$.

Fix such a general $P\in\Sigma_0$ and set $Z_0=Z$.
For $k\ge 0$ set $$Z_{k+1}=Z\setminus\{P_1, \ldots, P_{k+1}\}.$$
We claim inductively that for each $k$ the space of degree $d$ curves in $H$ containing $\pi_P(Z_k)$
has 
\begin{equation}\label{eq:dimstep}
\dim_{\mathbb C} [I(\pi_P(Z_k))]_d \ge k+2 .
\end{equation}

For $k=0$, this is exactly the assumption that $\pi_P(Z)$ lies on a pencil of curves of degree~$d$.

Assume~\eqref{eq:dimstep} holds for $k$.
Since $Z$ is general in $\PP^3$ and $P$ is a general point of $\Sigma_0$,
the projected point $\pi_P(P_{k+1})$ is a general point of $H$.

Since we impose this condition only for finitely many indices $k$, we may choose $P$ general in $\Sigma_0$ so that it holds for all $k.$
 In particular,   we may assume that $\pi_P(P_{k+1})$ is not contained in every curve of the linear system
$V_{k+1}=[I(\pi_P(Z_{k+1}))]_d$; equivalently, there exists a curve $C\in V_{k+1}$ with
$\pi_P(P_{k+1})\notin C$.
Hence the condition of passing through $\pi_P(P_{k+1})$ imposes one independent linear condition on
$V_{k+1}$, so the subspace of curves in $V_{k+1}$ passing through $\pi_P(P_{k+1})$ has dimension
$\dim V_{k+1}-1$.

If $\dim_{\mathbb C} [I(\pi_P(Z_{k+1}))]_d=k+2$, then requiring vanishing at the additional point
$\pi_P(P_{k+1})$ would impose (for a general choice) one independent linear condition, giving
\[
\dim [I(\pi_P(Z_k))]_d \le \dim [I(\pi_P(Z_{k+1}))]_d - 1 = k+1,
\]
contradicting the inductive hypothesis.
Therefore $\dim [I(\pi_P(Z_{k+1}))]_d\ge k+3$, proving~\eqref{eq:dimstep} for $k+1$.

In particular, the set $Z_{N-1}$ consists of a single point, $P_N$,
and the inequality~\eqref{eq:dimstep} gives
\[
\dim [I(\pi_P(Z_{N-1}))]_d \ge N+1 = \binom{d+2}{2}+1.
\]
But $[I(\pi_P(Z_{N-1}))]_d \subset [R_H]_d$ and $\dim [R_H]_d = N=\binom{d+2}{2}$,
so this is a contradiction.
\end{proof}

We now apply Theorem~\ref{t. not a pencil} to the case of general sets of $\binom{d+2}{2} \pm 1$ points in $\PP^3$. 


\subsection{$d$-Weddle locus for $\binom{d+2}{2} +1$ general points}\label{sec:+1}

Let $Z$ be a set of $N+1=\binom{d+2}{2}+1$ general points in $\PP^3$, $d\ge 2$. 

Since $Z$ is general, it imposes independent conditions on forms of degree $d$, hence
\[
\dim [I(Z)]_d
= \binom{d+3}{3}-(N+1)
= \binom{d+2}{3}-1.
\]
We are interested in the locus of points $P\in\PP^3\setminus Z$ for which there exists a degree $d$ cone with vertex $P$
containing $Z$, equivalently for which
\[
\dim [I(Z)\cap I(P)^d]_d \ge 1.
\]
Using the Macaulay duality exact sequence~\eqref{OrigMacDualSeqPrime}, this condition is detected by the failure of maximal rank of a
$\bigl(\binom{d+2}{3}-1\bigr)\times \binom{d+2}{3}$ matrix of linear forms in the coordinates of $P$.
Consequently, the $d$-Weddle scheme is defined by the maximal minors of this matrix; when these minors can cut out a curve (codimension $2$), Lemma~\ref{J1 lemma} applies and forces both the degree and the arithmetically Cohen--Macaulay property.

Emch considers this situation in \cite[page 278]{EMCH} and proposes a closed formula for the degree of the resulting $d$-Weddle locus:
\[
\frac{1}{72} ( 2d^6 + 12 d^5 + 17 d^4 - 66 d^3 - 271 d^2 + 954d - 648).
\] 
His formula gives the correct value in some small cases (for instance, when $d=2$ it yields degree $6$), but it is not correct in general. Already for $d=3$ (so $N+1=\binom{5}{2}+1=11$ general points) it predicts a curve of degree $52$; instead, our formula in Proposition \ref{binom d+2 2 + 1} shows that the $3$-Weddle locus is a curve of degree $45$, and this agrees with computations in symbolic algebra software.

We now explain the general pattern and record the correct degree formula. Consider the Macaulay duality exact sequence~\eqref{OrigMacDualSeqPrime}:
{\footnotesize\[
\left[\frac{R}{(L_{P_1}^{d},\ldots,L_{P_{N+1}}^{d})}\right]_{d-1}
\stackrel{\times L_P}{\longrightarrow}
\left[\frac{R}{(L_{P_1}^{d},\ldots,L_{P_{N+1}}^{d})}\right]_{d}
\longrightarrow
\left[\frac{R}{(L_{P_1}^{d},\ldots,L_{P_{N+1}}^{d},L_P)}\right]_{d}
\to 0,
\]
}where $L_P$ is linear form corresponding to the point $P\in\PP^3$.
The first vector space has dimension
$\binom{d+2}{3},$
while the second has dimension
\[
\dim \left[\frac{R}{(L_{P_1}^{d},\ldots,L_{P_{N+1}}^{d})}\right]_{d}
=\binom{d+3}{3}-(N+1)=\binom{d+2}{3}-1.
\]
Thus multiplication by $L_P$ is represented (after choosing bases) by a
\[
\left(\binom{d+2}{3}-1\right)\times \binom{d+2}{3}
\]
matrix of linear forms, and the $d$-Weddle scheme is cut out by its maximal minors. Because of the size of the matrix, the maximal minors define a locus of codimension one or two.
In particular, the expected codimension is $2$, and when this expected codimension occurs, Lemma~\ref{J1 lemma} determines the degree and implies the scheme is ACM.

For $d=2$ this reproduces the fact that the $2$-Weddle locus of seven general points is a curve of degree $6$ (indeed ACM of genus $3$, cf.\ Proposition~\ref{seven pts} and the discussion below).
For $d=3$ we have a $9\times 10$ matrix of linear forms. The $3$-Weddle surface for any subset of $10$ of the $11$ points has degree $10$, and two such surfaces arising from two different $10$-subsets share no component by Theorem~\ref{t. not a pencil} and generality. Hence the $3$-Weddle scheme for $11$ general points has codimension at least $2$, hence exactly $2$. So, by Lemma~\ref{J1 lemma}, it is an ACM curve of degree
$\binom{10}{2}=45.$

By the same reasoning, we have the following result.

\begin{proposition}\label{binom d+2 2 + 1} 
The $d$-Weddle locus for a general set of $\binom{d+2}{2}+1$ points in $\PP^3$ is an ACM curve of degree
\[
\binom{ \binom{d+2}{3} }{2} = \frac{1}{72} d^6 + \frac{1}{12} d^5 + \frac{13}{72} d^4 + \frac{1}{12} d^3 - \frac{7}{36} d^2 - \frac{1}{6} d.
\]
\end{proposition}


\subsection{$d$-Weddle locus for $\binom{d+2}{2} -1$ general points}

Consider the analogous question for  $Z=\{P_1,\ldots,P_{N-1}\},$ a set of $N-1=\binom{d+2}{2}-1$ general points in $\PP^3$.

 It is obvious that the general projection lies on a plane curve of degree $d$, so the question is to find the locus where the projection lies on a {\it pencil} of curves of degree $d$, i.e., for which $P$ we have $$\dim [I(\pi_P(Z))]_d \geq 2.$$ 

When $d=2$, i.e., $|Z|=5$, it is easy to see that  the 2-Weddle {\it locus} is precisely the union of lines joining two of the points, so the locus is a curve of degree $\binom{5}{2} = 10$. Then applying Proposition \ref{binom -1} we get that this is also the 2-Weddle scheme for the 5 points.

When $d\ge3$, the union of lines joining two points is  no longer  the full $d$-Weddle scheme of  $Z$. 
However, it is still contained in the
$d$-Weddle locus. Indeed, for example for $N-1$ general points, projecting from a point on a line joining two points of $Z$ gives $N-2$ points of $\PP^2$, which certainly lie on a pencil of curves of degree $d$.

In particular for $d=3$, i.e., $|Z|=9$, there are 36 such lines which form a reduced proper subscheme of the curve of degree 55 obtained in Proposition \ref{binom -1}.

\begin{proposition}\label{binom -1}
Let $Z$ be a set of $\binom{d+2}{2}-1$ general points in $\PP^3$. The expected codimension of the $d$-Weddle scheme is 2. If the $d$-Weddle scheme of $Z$ has the expected codimension then it is an ACM curve of degree
\[
\binom{\binom{d+2}{3}+1}{2}.
\]
This curve includes the $\displaystyle \binom{\binom{d+2}{2}-1}{2}$ lines joining two points of $Z$. In particular, it is never irreducible.
\end{proposition}

\begin{proof}
We again set $N = \binom{d+2}{2}$. 
By \eqref{OrigMacDualSeqPrime} we have the exact sequence
\[
[R/(L_1^d, \dots, L_{N-1}^d)]_{d-1} \stackrel{\times \ell}{\longrightarrow} [R/(L_1^d, \dots, L_{N-1}^d)]_d \rightarrow [R/(L_1^d,\dots,L_{N-1}^d, \ell)]_d \rightarrow 0.
\]
The first vector space has dimension $\binom{d+2}{3}$, and the second vector space has dimension $\binom{d+2}{3}+1$. Thus we have a $[\binom{d+2}{3}+1 ] \times \binom{d+2}{3}$ matrix of linear forms. By assumption, the ideal of maximal minors defines a scheme of codimension 2 in $\PP^3$, so 
Lemma \ref{J1 lemma} gives the ACMness and the degree of the $d$-Weddle scheme. 
The fact that it includes the lines joining two points of $Z$ is obvious, as mentioned above. 
\end{proof}

\begin{conjecture}
A set of $\binom{d+2}{2}-1$ general points in $\PP^3$ has $d$-Weddle scheme of the expected codimension, namely 2.
\end{conjecture}
\begin{remark}
    We believe that the conjecture is true. In particular, it holds for $d=2, 3, 4$. It is not hard to check the case $d=2$. The other two cases were verified by computer using random points, and the result holds for general points by semicontinuity.  
\end{remark}


\subsection{$d$-Weddle locus for $\binom{d+2}{2} \pm 2$ general points}
We first consider the case of $\binom{d+2}{2}+2$ general points in $\PP^3$, where, applying the previous section, one expects only finitely many
points of projection from which the image lies on a degree $d$ plane curve.

\begin{proposition}\label{general set plus 2}

The $d$-Weddle locus for a general set, $Z$, of $\binom{d+2}{2} +2$ points in $\PP^3$ is a zero-dimensional scheme of degree $\displaystyle \binom{\binom{d+2}{3}}{3}$.

\end{proposition}

\begin{proof}
Set $N=\binom{d+2}{2}$, so $|Z|=N+2$.
 The sequence \eqref{OrigMacDualSeqPrime} gives
\[
[R/(L_1^d, \dots, L_{N+2}^d)]_{d-1} \stackrel{\times \ell}{\longrightarrow} [R/(L_1^d, \dots, L_{N+2}^d)]_d \rightarrow [R/(L_1^d,\dots,L_{N+2}^d, \ell)]_d \rightarrow 0.
\]
Now the first vector space has dimension $q=\binom{d+2}{3}$ and the second has dimension 
\[
\binom{d+3}{3}-(N+2) = \binom{d+3}{3}-\binom{d+2}{2}-2 = \binom{d+2}{3}-2=q-2.
\]

Thus multiplication by $\ell$ is represented by a $(q-2)\times q$ matrix of linear forms on $\PP^3$.
The ideal of maximal minors therefore, by Lemma \ref{J1 lemma}, has expected codimension 3. If this is the correct codimension, then the degree (using again Lemma \ref{J1 lemma}) is the one claimed by the statement.

It remains to justify that the locus cannot have positive-dimensional components. Let $Z'\subset Z$ be any subset of $N+1$ points. 
By Proposition \ref{binom d+2 2 + 1},  the $d$-Weddle locus of  $Z'$ is an ACM curve, $C$; and by Theorem~\ref{t. not a pencil} a general point of any component of $C$ projects $Z'$ onto a set lying on a unique degree $d$ plane curve. Adding one further general point to $Z'$ to obtain $Z$ forces the projected additional point to miss that curve, so the same projection point cannot lie in the $d$-Weddle locus of $Z$.
Hence the $d$-Weddle locus of $Z$ is zero-dimensional, and the degree is as claimed.
\end{proof}

Now consider instead a general set of $N-2=\binom{d+2}{2}-2$ points in $\PP^3$.
The general projection lies on a pencil of plane curves of degree $d$, so the $d$-Weddle locus is the set of points in $\PP^3$ from which the projection lies on a 2 (projective) dimensional linear system. In the exact sequence \eqref{OrigMacDualSeqPrime} above, the second vector space now has dimension $\binom{d+2}{3}+2$. Again the dimensions differ by 2, so as in Lemma \ref{J1 lemma} we would again expect a zero-dimensional $d$-Weddle locus. However, this is not the case.

\begin{proposition}\label{general set minus 2}
Let $Z$ be a general set of $N - 2$ points in $\PP^3$. Then the $d$-Weddle locus does not have the expected codimension.
\end{proposition}

\begin{proof}
It is clear that the 1-skeleton of the $N-2$ points (i.e., the lines through pairs of the points) is contained in the locus. Since we expect codimension 3, we are done. 
\end{proof}

Something unusual happens in the case of 8 general points. We illustrate it with a careful analysis.


\subsection{The curious behavior of 8 general points}
Let $Z$ be a general set of 8 points in $\PP^3$. Since $8=\binom{2+2}{2}+2=\binom{3+2}{2}-2,$ the set $Z$ simultaneously falls into both the cases discussed above for $d=2$ and for $d=3$, respectively. This leads to a mixed behavior.

 For a general projection $\pi_P$, it is clear that the image $\pi_P(Z)$ lies on a pencil of cubics and does not lie on any conics.
\begin{itemize}
    \item { The $2$-Weddle scheme. } The linear system of quadrics through $Z$ has dimension
\[
\dim [I(Z)]_2=\binom{2+3}{3}-8=10-8=2,
\]
so it is a pencil. A point $P\in\PP^3$ lies in the $2$-Weddle locus precisely when the projection $\pi_P(Z)$ lies on a conic,
equivalently when there exists a quadric cone containing $Z$ with vertex $P$.
Thus the $2$-Weddle locus is the set of vertices of the singular quadrics in the pencil. Singular quadrics form a degree $4$ hypersurface in the $\PP^9$ parametrizing all quadrics in $\PP^3$
(the discriminant hypersurface). Since the pencil determined by a general $Z$ is a general line in $\PP^9$,
it meets the discriminant in $4$ reduced points. Hence the pencil contains exactly four quadric cones,
and therefore the $2$-Weddle locus consists of four distinct points (the corresponding vertices).
In particular, the $2$-Weddle scheme is a zero-dimensional scheme of degree $4$ supported on these four points; this agrees with Proposition \ref{general set plus 2}.

For general $Z$, these four vertices are disjoint from the $28$ lines joining pairs of points of $Z$.
Indeed, the condition that a vertex lie on a fixed secant line $\langle P_i,P_j\rangle$ is a closed incidence condition,
and a general choice of $Z$ avoids it. 
Equivalently, for random choices of $Z$ one checks by computer that no vertex lies on any of the $28$ secants; this persists for general $Z$ by semicontinuity.

\item { The $3$-Weddle scheme. } Note that the 28 lines forming the 1-skeleton of $Z$ are contained in its 3-Weddle locus, so
(as noted in Proposition \ref{general set minus 2}) the $3$-Weddle scheme of $Z$ does not have the expected codimension.  

Notice that any point in the 2-Weddle locus of $Z$ must also lie in the 3-Weddle locus of $Z$. Indeed, if $P$ is a point on the 2-Weddle locus of $Z$, the projection of $Z$ from $P$ is a set of 8 distinct points on a conic.  By generality, no projection of $Z$ contains 4 or more collinear points so the eight points cannot be contained in two lines, hence
this conic must be smooth. 
Thus the projection $\pi_P(Z)$ must in fact be a complete intersection, and $\dim [I({\pi_P (Z)})]_3 = 3$ so $P$ lies in the $3$-Weddle locus. Thus the 3-Weddle locus has a curve component of degree 28 plus a zero-dimensional scheme of degree 4, so it is not unmixed. 
Our computations indicate that the 3-Weddle locus contains nothing else.
\end{itemize}

The next example shows that even for 8 points in LGP, the 2-Weddle scheme can be nonreduced of degree 4.

\begin{example} \label{nonred sup on pt}
Let $C\subset \PP^3$ be a twisted cubic. Choose seven points
$P_1,\ldots,P_7\in C$, and let $P_8$ be a general point on the
tangent line $\lambda=T_{P_7}C$ at $P_7$. Set
\[
Z=\{P_1,\ldots,P_8\}\subset \PP^3.
\]

It is easy to check that the set $Z$ is in LGP. For the same reason as above (namely, the quadrics through $Z$ form a pencil and a pencil contains
four singular members counted with multiplicity), the $2$-Weddle scheme is $0$-dimensional of degree $4$.

Let $P$ be a point of projection in the $2$-Weddle locus of $Z$. 

By Proposition \ref{seven pts}, $P$ cannot lie on a line joining two points of $\{ P_1,\dots,P_7 \}$. 

Moreover, if $P\notin C$ then $\pi_P(C)$ is a plane cubic and $\pi_P(Z)$ consists of 7 distinct points; therefore $\pi_P(P_1),\ldots,\pi_P(P_7)$ cannot all lie on a conic. 

Thus necessarily $P\in C$. In order for the eighth point to lie on the same conic, we must also have that
$\pi_P(P_8)$ lies on $\pi_P(C)$.
But $P_8\in \lambda=T_{P_7}C$, so $\pi_P(P_8)\in \pi_P(C)$ holds if and only if  $P$ lies on the same line $\lambda$ . Hence any such $P$ must lie in
\[
C\cap \lambda=\{P_7\}.
\]
Therefore the $2$-Weddle locus of $Z$ is the single point $P_7$.

Since the $2$-Weddle scheme has degree $4$ but its support is the single point $P_7$,
it follows that the $2$-Weddle scheme is nonreduced and supported at $P_7$.

\end{example}

\section{Connections to Lefschetz Properties}\label{sec:LefschetzConnections}

In this section we record some observations relating the $d$-Weddle scheme to objects
arising in the study of Lefschetz properties. This connection is motivated in particular by Subsection~\ref{not max minors}.

Given a graded artinian ring $A=\oplus_{i\geq0} A_i$ (such as $A=R/(L_1^d,\dots,L_r^d)$ for linear forms $L_j$ spanning 
$[R]_1$), recall that $A$ has the {\it Weak Lefschetz Property} (WLP) 
if there is a linear form $L$ such that the map $A_t\to A_{t+1}$ given by
multiplication $\times L$ has  maximal rank, for every $t$.
And $A$ has the {\it Strong Lefschetz Property} (SLP) if there is a linear form $L$ for which the map $A_t\to A_{t+s}$ given by
multiplication $\times L^s$ has maximal rank, for all $t$ and $s$; see \cite{MN-Tour}. Either of these is an open condition, by semicontinuity, so to show that either of these properties fails it suffices to find a $t$ (and $s$ for SLP) for which maximal rank fails for a general $L \in [R]_1$. When WLP does hold, one can still study the {\it non-Lefschetz locus}  \cite{BMMN}, which is the set of linear forms (with a natural scheme structure) for which maximal rank fails, even if it does hold for the general $L$.

Now, fix a set of points $Z = \{P_1,\ldots, P_r\}$ and let $L_1,\dots,L_r$ be the dual linear forms. Let $P$ be another point, and
let $L\in [R]_1$ be its dual linear form. 
The bridge to Lefschetz theory is provided by the exact sequence \eqref{OrigMacDualSeqPrime}, which for convenience we recall here with the current notation. 
\[
[R/(L_1^d,\dots,L_r^d)]_{d-1} \stackrel{\times L}{\longrightarrow} [R/(L_1^d,\dots,L_r^d)]_d \rightarrow [R/(L_1^d,\dots,L_r^d, L)]_d \rightarrow 0.
\]
Ideals generated by powers of linear forms have been extensively studied from the perspective of Lefschetz properties; 
 (see for instance \cite{DIV}, \cite{FM}, \cite{HSS}, \cite{MR-SLP}, \cite{MMN}, \cite{MN-L2},  \cite{SS} to name a few). Furthermore, \cite[Proposition 2.17]{HMNT} relates unexpected hypersurfaces for a set of points to the failure of the Strong Lefschetz Property for the corresponding quotient by the ideal of powers of linear forms. Most of these  papers use Macaulay duality as an important tool, so it is not new that the above sequence is related to Lefschetz properties. Here we make explicit the link between our $d$-Weddle schemes and the non-Lefschetz loci of \cite{BMMN}.

By definition, the last graded piece in the sequence above measures the failure of surjectivity of multiplication by $L$ from degree $d-1$ to degree $d$:
\[
\dim_\field [R/(L_1^d,\dots,L_r^d, L)]_d
=
\dim_\field \coker\!\left(
[R/(L_1^d,\dots,L_r^d)]_{d-1} \xrightarrow{\times L} [R/(L_1^d,\dots,L_r^d)]_d
\right).
\]
On the other hand, we have already observed that this same dimension is equal to
$\dim_\field [I(P_1)\cap\cdots\cap I(P_r)\cap I(P)^d]_d$,
i.e., to the dimension of the space of degree $d$ cones with vertex $P$ containing $Z$.
Consequently, this dimension is larger than the generic value if and only if the multiplication map $\times L$ drops rank compared to its generic behavior. 
Equivalently, the locus of points $P$ where $Z$ admits ``more than expected'' degree $d$ cones with vertex $P$ is precisely the locus of linear forms $L$ for which the map
$[R/(L_1^d,\dots,L_r^d)]_{d-1}\xrightarrow{\times L}[R/(L_1^d,\dots,L_r^d)]_d$
fails to have maximal rank.

In particular, the support of the $d$-Weddle scheme coincides with the non-Lefschetz locus (in the sense of \cite{BMMN}) for multiplication from degree $d-1$ to degree $d$ in the quotient $R/(L_1^d,\dots,L_r^d)$, and the determinantal scheme structure we place on the $d$-Weddle locus agrees with the non-Lefschetz scheme structure of \cite{BMMN}.
Thus, from this perspective, $d$-Weddle schemes provide a geometric incarnation of the non-Lefschetz locus associated to uniform powers of linear forms.

Recall from \cite{HMNT} that a set of points $Z$ in $\mathbb P^n$  admits an unexpected cone of degree $d$ if 
\[\dim_{\mathbb C}[I(Z)\cap I(P)^d]_d>\max\left\{0,\dim_{\mathbb C}[I(Z)]_d- \binom{n+d-1}{n}\right\},\]
where $P$ is a general point in $\mathbb P^n$.

Equivalently, using Macaulay duality as described in Section \ref{sec:Macaulay duality},  $Z$ admits unexpected cones in degree $d$ precisely when the multiplication map
$[R/(L_1^d,\dots,L_r^d)]_{d-1}\xrightarrow{\times L}[R/(L_1^d,\dots,L_r^d)]_{d}$
fails to have maximal rank for a general linear form $L$.

This allows one to translate results on unexpected cones into failures of WLP for quotients by powers of linear forms.

We have from \cite[Theorem 3.5]{CM} that $(a,b)$-grids have unexpected cones in degree $a$, and also in degree $b$ provided $a,b \geq 3$.  
We have the following result.

\begin{corollary}
Let $Z$ be an $(a,b)$-grid with $b \geq a\geq 2$ and $b \geq 3$, and let $L_1,\dots,L_{ab}$ be the dual linear forms. Then $R/(L_1^a,\dots,L_{ab}^a)$ fails the Weak Lefschetz Property from degree $a-1$ to degree $a$, and if $b \geq a \geq 3$ then  $R/(L_1^b,\dots,L_{ab}^b)$ fails the Weak Lefschetz Property from degree $b-1$ to degree~$b$.
\end{corollary}

\section{Future directions and open problems}\label{sec:Future}
\subsection{Projecting 8 points in $\mathbb P^4$ to a complete intersection}

For a general set $Z$ of 8 points in $\PP^4$ there is a somewhat mysterious  arithmetically Cohen-Macaulay (ACM) curve $C$ of degree 7 in $\PP^4$ such that the projection  of $Z$ from any point of $C$ {\it is} a complete intersection of quadrics in $\mathbb P^3$.  

\begin{remark} \label{proj8fromP4}
 Let $Z$ be a general set of 8 points in $\PP^4$.  Using the computer with random points we found that the 2-Weddle locus has the expected codimension 3. More precisely, there is a curve from the points of which $Z$ projects to a set of points lying on a net of quadrics in $\PP^3$.

Applying the methods used in this paper, especially the exact sequence \eqref{OrigMacDualSeqPrime}  and Lemma~\ref{J1 lemma}, we see that the 2-Weddle locus is a curve of degree 35. However, it is clear that the 28 lines joining two points of $Z$ are all components of this locus. What remains is a curve of degree~7. 
We have confirmed by computer calculation that this curve is smooth of genus 3 and is arithmetically Cohen-Macaulay, lying on a surface of degree 3. But we do not see a geometric reason for this fact, or see how to visualize this curve from a geometric point of view.
\end{remark}

\subsection{Generalizing the Weddle condition} 

Given a finite set of points $Z\subset \PP^n$ and projections $\pi_Q\colon\PP^n\dashrightarrow \PP^{n-1}$, one can replace the projection-to-hypersurface condition in the definition of Weddle loci by the requirement that $\pi_Q(Z)$ lie on a subvariety of prescribed type, leading to new families of exceptional loci of projection centers.

For example, a natural extension of the Weddle philosophy is to replace
projection from $\PP^3$ to a plane conic by projection from $\PP^n$ to a rational normal curve of degree $n-1$ in $\PP^{n-1}$. This is done in Appendix \ref{app: GO}, contributed by G.\ Ottaviani.

\subsection{Projecting from $\mathbb P^n$ to $\mathbb P^m$}
For a set of points $Z$ in $\mathbb P^n$ one can define an analogous Weddle locus/scheme by looking at the projection map 
$$\pi_L\colon \mathbb P^n \dashrightarrow \mathbb P^m$$
for $m<n-1$ where $L$ is a linear space of projective dimension $n-1-m$.

The first case of interest would be projecting 6 points in $\mathbb P^4$ in LGP to $\mathbb P^2$.

For $Z\subset \PP^4$ a set of six points in linear general position, consider projection
$\pi_L\colon \PP^4\dashrightarrow \PP^2$ from a line $L\simeq \PP^1$, so that $L$ varies in
the Grassmannian $\GGr(1,4)$. For a general choice of $L$ avoiding $Z$, the image
$\pi_L(Z)\subset \PP^2$ is a set of six points in general position, hence it does \emph{not}
lie on any conic. Therefore the condition ``$\pi_L(Z)$ lies on a conic'' is a genuine
degeneracy condition.

This suggests a Grassmannian variant of the Weddle construction: define $\mathcal W\subset \GGr(1,4)$
to be the closure of the locus of lines $L$ such that
$h^0(\PP^2,\mathcal I_{\pi_L(Z)}(2))\ge 1$, i.e.,\ such that the six projected points satisfy an
unexpected quadratic relation, namely they lie on a conic. One may further consider the locus where
$h^0(\PP^2,\mathcal I_{\pi_L(Z)}(2))\ge 2$, corresponding to a pencil of conics through $\pi_L(Z)$.
It would be interesting to describe the geometry (and, e.g., degree in Pl\"ucker coordinates)
of such determinantal loci on $\GGr(1,4)$.

Some initial investigation suggests that $\mathcal{W}$ is defined by a section in the fourth power of the ample generator of $\Pic(\GGr(1,4))$, i.e., it is a subvariety of dimension $5$ and degree $20$ in $\PP^9$, where the Grassmanian is embedded by the Pl\"ucker coordinates. 

\subsection{Positive characteristic}
Throughout this paper we work over $\CC$, but the definition of the $d$-Weddle locus (and the
associated $d$-Weddle scheme) makes sense over any algebraically closed field, including
fields of positive characteristic.
It would be interesting to investigate which new phenomena occur in this setting.

Two natural sources of characteristic-dependent behavior are:
(i) the use of differential operators in the interpolation--matrix description (which may require
divided powers in characteristic $p$), and
(ii) the behavior of the Lefschetz properties for algebras of the form
$R/(L_1^d,\ldots,L_r^d)$, which is known to be sensitive to the characteristic.
In particular, one may ask when the determinantal scheme structures considered here remain
reduced (or equidimensional), how their degrees compare with the characteristic-zero formulas,
and whether new families of nonreduced or higher-multiplicity components appear.

\subsection{The singular locus of the Weddle scheme of a general set of points}

We have seen that a general set of $\binom{d+n}{d}$ points in $\mathbb P^{n+1}$ has a $d$-Weddle scheme $\Sigma$ that is a hypersurface of degree $\binom{d+n}{n+1}$ (Theorem \ref{WeddleSchemeSigmaThm}). What is the singular locus of this hypersurface?

In Remark \ref{Remark4.9} and Appendix \ref{AppendixA} we show that the singular locus of the Weddle surface associated to 6 points in linear general position in $\PP^3$ ($d=n=2$)  consists precisely of those six points.  Experimentally we have verified (over a large finite field) that the singular locus of the 2-Weddle hypersurface for 10 ``random" points in $\mathbb P^4$ consists of 60 distinct points. These consist of the 10 original points together with 50 additional points coming from the fact that the subminors of a general $5 \times 5$ matrix of linear forms is a reduced zero-dimensional scheme of degree 50 (although we still do not understand the corresponding geometric connection between these 50 singular points and the original set of points). For 15 ``random" points in $\mathbb P^5$ we have verified experimentally (again, over a large finite field)  that the singular locus of the 2-Weddle hypersurface includes a curve of degree 105 defined by  the submaximal minors, and the 15 original points.

Given a general set of $\binom{d+n}{d} $ points in $\PP^{n+1}$, what is the dimension and what is the degree of the singular locus of the $d$-Weddle hypersurface? How does its geometry relate to the original set of points? Notice that whenever the codimension of the singular locus of $\Sigma$ is $\geq 2$, $\Sigma$ must be irreducible.

\vspace{5pt}

\noindent{\bf Acknowledgements:} Chiantini and Favacchio are members of the Italian GNSAGA-INDAM. The work of the Favacchio was supported by the funding PREMIO\_SINGO\-LI\_RIC\_[2025] from the Department of Engineering, University of Palermo. Migliore was partially supported by Simons Foundation grant \#839618. All computations were performed using  \cite{CoCoA, DGPS, M2}. And we thank Giorgio Ottaviani for contributing Appendix~\ref{app: GO}.

\appendix
\section{Singular locus of Weddle surface}\label{AppendixA}
{\setstretch{1}
Here is Macaulay2 code for computing the singular locus of the Weddle surface for six points in LGP. 
We may choose coordinates such that the points are $[0:0:0:1]$, $[0:0:1:0]$, $[0:1:0:0]$, $[1:0:0:0]$, $[1:1:1:1]$ and $[a:b:c:d]$, where
by linear general position $[a:b:c:d]$ is not on any plane defined by any of the other three points and so $abcd(a-b)(a-c)(a-d)(b-c)(b-d)(c-d)\neq0$. The result is that
the singular locus, as given by the primary decomposition $P$, consists of the six points.

\begin{verbatim}
R=QQ[a,b,c,d,x,y,z,w];
H=d*(b-c)*x^2*y*z+d*(c-a)*x*y^2*z+d*(a-b)*x*y*z^2+c*(d-b)*x^2*y*w+
c*(a-d)*x*y^2*w+b*(c-d)*x^2*z*w+a*(d-c)*y^2*z*w+b*(d-a)*x*z^2*w+
a*(b-d)*y*z^2*w+c*(b-a)*x*y*w^2+b*(a-c)*x*z*w^2+a*(c-b)*y*z*w^2;
I1=ideal(diff(x,H),diff(y,H),diff(z,H),diff(w,H));
I2=time I1:ideal(a);
I3=time I2:ideal(b);
I4=time I3:ideal(c);
I5=time I4:ideal(d);
I6=time radical(I5);
I7=time I6:ideal(a-b);
I8=time I7:ideal(a-c);
I9=time I8:ideal(a-d);
I10=time I9:ideal(b-c);
I11=time I10:ideal(b-d);
I12=time I11:ideal(c-d);
P=time primaryDecomposition I12;
\end{verbatim}
}


\section{The Weddle locus to a rational normal curve\for{toc}{. \em (Giorgio Ottaviani)}}\label{app: GO}

\begin{center}
\textit{By Giorgio Ottaviani}
\end{center}

We compute the closure of the locus of
centers of projection from which $n+3$ points in $\PP^n$ in linear general position map to  $n+3$ points 
lying on a rational normal curve in $\PP^{n-1}$. It is a surface of degree $2^n-n-1$.
The proof is an application of a result obtained in collaboration with Rekha Thomas in \cite{OT}.

\subsection{Invariant theory of $2\times n\times (n+1)$ tensors}\label{subsec:2n(n+1)}
 The content of this subsection goes back to Weierstrass. 
 The space $\PP(\CC^2\otimes\CC^n\otimes\CC^{n+1})$ is a prehomogeneous space under the action
 of $SL(2)\times SL(n)\times SL(n+1)$; this means there is a dense orbit, see \cite{M}. The space is prehomogeneous even under the action of the subgroup 
 $SL(n)\times SL(n+1)$, but we do not need this fact. The complement of the dense orbit is the hypersurface defined by the hyperdeterminant,
 of degree $n(n+1)$, which is the generator of the invariant ring $\CC[\CC^2\otimes\CC^n\otimes\CC^{n+1}]^{
 SL(2)\times SL(n)\times SL(n+1)}$, see \cite[chap. 14, \S 3]{GKZ} and \cite{O} and it is a regular invariant according to \cite[\S 4.2]{M}.
It is a boundary format hyperdeterminant.
 A tensor with nonzero hyperdeterminant is called nondegenerate.
 \begin{proposition}\label{prop:1generic} Let $t\in \CC^2\otimes\CC^n\otimes\CC^{n+1}$ with entries $t_{ijk}$ for $i=0,1$,
 $j=0,\ldots, n-1$, $k=0,\ldots, n$. The following are equivalent
 \begin{itemize}
 \item $t$ is nondegenerate
 \item $t$ is $1$-generic, namely the two linear forms $\sum_{k=0}^nt_{0,j,k}x_k$ and $\sum_{k=0}^nt_{1,j,k}x_k$
 are independent for $j=0,\ldots, n-1$.
 \end{itemize}
 \end{proposition}
 \begin{proof}
 \cite[Prop. 9.4]{Har}, joint to the fact that the tensor space is prehomogeneous. 
 \end{proof}
 
 \subsection {Equations of rational normal curves and their Weddle locus}
 Denote by ${\mathcal E}_n$ the set of $n+2$ points in $\PP^n$ with coordinates $\{e_0,\ldots, e_n,\sum_{i=0}^ne_i\}$. Every set of $n+2$  points in LGP coincides with ${\mathcal E}_n$ after a linear change of coordinates.

 The following is an expansion of \cite[Remark 2 \S 4]{OT}.
 \begin{lemma}\label{lem:ratnorm}
  Let $a=\begin{pmatrix}a_0\\ \vdots\\a_n\end{pmatrix}\in\PP^n$ such that ${\mathcal E}_n\cup\{a\}$ is in LGP.
 Let $b=\begin{pmatrix}b_0\\ \vdots\\b_n\end{pmatrix}$ be a point in $\PP^n$.
 Let $A$ be the $(n+1)\times 3$ matrix with $i$-th row $$(a_i, b_i, a_ib_i)\textrm{\ \ for\ }i=0,\ldots, n.$$
 The rational normal curve through the $n+3$ points ${\mathcal E}_n\cup\{a\}$ can be described as the variety
$$\{b\in\PP^n \ | \ \mathrm{rk}(A)\le 2\}.$$
\end{lemma}
\begin{proof}
By assumption $a_0\neq 0$.
Gaussian elimination on the first row reduces the condition on the rank of $A$ to
\begin{equation}\label{eq:ratab}\mathrm{rk}\begin{pmatrix}a_0b_1-a_1b_0&a_1(b_1-b_0)\\
\vdots&\vdots\\
a_0b_i-a_ib_0&a_i(b_i-b_0)\\
\vdots&\vdots\\
a_0b_n-a_nb_0&a_n(b_n-b_0)\\
\end{pmatrix}\le 1,\end{equation}
which is a condition on $b\in\PP^n$.
This is a $n\times 2$ matrix with entries linear in $b$, hence it is a 
$2\times n\times (n+1)$ tensor as in \ref{subsec:2n(n+1)}.
The assumption that ${\mathcal E}_n\cup\{a\}$ is in LGP implies easily that the tensor is $1$-generic according to
Prop. \ref{prop:1generic}.
Hence the condition  (\ref{eq:ratab}) defines a rational normal curve, see \cite[Example 9.3]{Har}.
It is straightforward to check that (\ref{eq:ratab}) is satisfied when $b\in {\mathcal E}_n\cup\{a\}$,
hence the rational normal curve contains  ${\mathcal E}_n\cup\{a\}$ .
The thesis follows by the uniqueness of the rational normal curve through $n+3$ general points.
\end{proof}

\begin{remark} The condition of Lemma \ref{lem:ratnorm} is symmetric in $a$, $b$ and can be interpreted as the condition
that the $n+4$ points  ${\mathcal E}_n\cup\{a, b\}$ lie on a rational normal curve.
\end{remark}

\begin{example}
For $n=3$ the hyperdeterminant of the tensor in (\ref{eq:ratab}) can be computed according to
\cite[Chapt. 14 Theor. 3.3]{GKZ} or \cite[Theor. 4]{O} and it is (up to sign)
$a_0^3a_1a_2a_3\prod_{0\le i<j\le 3}(a_i-a_j)$, which vanishes if and only if ${\mathcal E}_3\cup\{a\}$ is not in LGP. 
\end{example}

The following is the main result of this appendix.

\begin{theorem}\label{thm:weddlerat}
 Let $a=\begin{pmatrix}a_0\\ \vdots\\a_n\end{pmatrix}\in\PP^n$ such that ${\mathcal E}_n\cup\{a\}$ is in LGP.
  Let $b=\begin{pmatrix}b_0\\ \vdots\\b_n\end{pmatrix}$ be a point in $\PP^n$.
 Let $N$ be the $n\times 3$ matrix with $i$-th row $$((a_i-a_n)b_i, a_nb_i-a_ib_n, (b_i-b_n)(a_nb_i-a_ib_n))\textrm{\ \ for\ }i=0,\ldots, n-1.$$
 (i) The Weddle locus  given by the closure of the locus of
centers of projection from which ${\mathcal E}_n\cup\{a\}$ maps to  $n+3$ points 
lying on a rational normal curve in $\PP^{n-1}$ can be described as the variety
$$\{b\in\PP^n \ | \ \mathrm{rk}(N)\le 2\}.$$
(ii) The Weddle locus in (i) is a surface of degree $2^n-n-1$.
\end{theorem}
\begin{proof} The projection $\pi_b$ from $b$ to the hyperplane $b_n=0$ satisfies

$\pi_b(e_i)=e_i$ for $i=0,\ldots, n-1$, $\pi_b(e_n)=\begin{pmatrix}b_0\\
\vdots\\
b_{n-1}\end{pmatrix}$,

 $\pi_b(\sum_{i=0}^ne_i)=\begin{pmatrix}
\vdots\\
b_{i}-b_n\\
\vdots\end{pmatrix}$,
$\quad\pi_b(a)=\begin{pmatrix}\vdots\\
b_ia_n-b_na_i\\
\vdots\end{pmatrix}$.
With a linear change of coordinates $x_i'=x_i/b_i$ for $i=0,\ldots, n-1$, these images
become
${\mathcal E}_{n-1}, \begin{pmatrix}\vdots\\(b_i-b_n)/b_i\\ \vdots\end{pmatrix}$, $\begin{pmatrix}\vdots\\(b_ia_n-b_na_i)/b_i\\ \vdots\end{pmatrix} $.
 
 By Lemma \ref{lem:ratnorm}, these points lie on a rational normal curve in $\PP^{n-1}$ if and only if,
 given 
 \begin{equation}
 B=\begin{pmatrix}
 \vdots&\vdots&\vdots\\
 (b_i-b_n)b_i&(b_ia_n-b_na_i)b_i&(b_i-b_n)(b_ia_n-b_na_i)\\
 \vdots&\vdots&\vdots\\
 \end{pmatrix}\end{equation}
 we have $\mathrm{rk}(B)\le 2$.
 
 Call $B=\begin{pmatrix}B_0&B_1&B_2\end{pmatrix}$
 and perform the following column operation
 $\begin{pmatrix}B_0a_n-B_1&B_1-B_2&B_2\end{pmatrix}.$
 After a straightforward computation, get the $n\times 3$ matrix with $i$-th row $$(b_n(a_i-a_n)b_i, b_n(a_nb_i-a_ib_n), (b_i-b_n)(a_nb_i-a_ib_n))\textrm{\ \ for\ }i=0,\ldots n-1.$$ Dividing by $b_n$ the first two columns, get (i). Note that every maximal minor of $N$ has
still $a_n$ as a factor. Removing such a factor, each maximal minor of $N$ has bidegree $(2,4)$ in $(a, b)$.

 In order to prove (ii), consider that the Weddle locus is the degeneracy locus of a bundle map on $\PP^n$
 $$\mO(-1)^2\oplus\mO(-2)\to\mO^n.$$
 
 By Giambelli-Thom-Porteous formula we wish the coefficient of $t^{n-2}$  in the ratio of Chern polynomials $\frac{1}{(1-t)^2(1-2t)}$.
 
We compute
 $$\frac{1}{(1-t)^2(1-2t)}= \frac{-2}{1-t}-\frac{1}{(1-t)^2}+\frac{4}{1-2t}=\sum_{n\ge 0}\left[(-2)-(n+1)+2^{n+2}\right]t^n$$
 which implies (ii). 
\end{proof}

\begin{remark} Denote by $S_n$ the Weddle surface in $\PP^n$ computed in Theorem \ref{thm:weddlerat}. It seems that the points in ${\mathcal E}_n\cup\{a\}$ have multiplicity $n-1$
and are the only singular points of $S_n$. An insight by Luca Chiantini is that projecting $S_n$ by one of its singular points we get a $2:1$ ramified covering of $S_{n-1}$.
Note this insight is compatible with the equality $\frac{1}{2}\left( \deg S_n-(n-1)\right) = \deg S_{n-1}$, which is verified by (ii) of Theorem \ref{thm:weddlerat}.
\end{remark}

\end{document}